\documentclass[preprint,3p,times]{elsarticle}

\usepackage[hidelinks]{hyperref}
\usepackage{amsmath,amssymb}
\usepackage{mathtools}	% for \coloneqq
\usepackage{stmaryrd} % Brackets
\usepackage{bm}		  % bold math symbols
\usepackage{accents}
\usepackage{cancel}
\usepackage[toc,page]{appendix}
\usepackage{subfigure}
\usepackage{booktabs}
\usepackage{nicefrac}
\usepackage[T1]{fontenc}
\usepackage[utf8]{inputenc}

\usepackage[dvipsnames]{xcolor}
\definecolor{chcol}{rgb}{0.4,0.,0.9}
\newcommand{\changeOne}[1]{\textcolor{black}{#1}}
\newcommand{\changeTwo}[1]{\textcolor{black}{#1}}
\newcommand{\changeThree}[1]{\textcolor{black}{#1}}
\newcommand{\changeOther}[1]{\textcolor{black}{#1}}
\usepackage[ruled]{algorithm2e}
\DontPrintSemicolon

\newcommand{\avg}[1]{\left\{\hspace*{-1pt}\left\{#1\right\}\hspace*{-1pt}\right\}}
\newcommand{\jump}[1]{\ensuremath{\left\llbracket #1 \right\rrbracket}}

\newcommand\statevec[1]{\mathbf #1}					% state vector, e.g. [rho, rhov, E, B]^T
\newcommand\statemat[1]{\underline{#1}}
\newcommand\polydeg{p}	% polynomial degree
\newcommand\statedim{n} % state dimension
\newcommand\stateset{\mathbb{R}^{\statedim}}
\newcommand\ncpath[3]{\Phi(#1; #2, #3)}

\newtheorem{theorem}{Theorem}
\newtheorem{lemma}{Lemma}
\newdefinition{remark}{Remark}
\newdefinition{definition}{Definition}
\newproof{proof}{Proof}

\journal{Journal of Computational Physics}

\newcommand{\orcid}[1]{\href{https://orcid.org/#1}{\includegraphics[width=10pt]{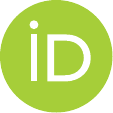}}}

\begin{document}

\begin{frontmatter}

\title{A new class of entropy stable fluctuations for the discontinuous Galerkin method with application to the Saint-Venant-Exner model}

\author[liu]{Patrick Ersing\corref{cor1}\orcid{0009-0005-3804-5380}}
\ead{patrick.ersing@liu.se}
\cortext[cor1]{Corresponding author}

\author[liu]{Andrew R. Winters\orcid{0000-0002-5902-1522}}

\affiliation[liu]{organization={Linköping University, Department of Mathematics},
            city={Linköping},
            postcode={58183}, 
            country={Sweden}}

\begin{abstract}
In this work we consider entropy stable discontinuous Galerkin methods applied to nonconservative hyperbolic systems.
We introduce a new class of entropy conservative fluctuations that allow us to construct entropy conservative schemes without any system-specific derivations.
We \changeOne{also} demonstrate that a loss of entropy symmetrization for nonconservative systems restricts the design of entropy stable fluctuations and propose a novel blending procedure to construct entropy stable dissipation terms from general numerical viscosity matrices.
The resulting methodology is applied to develop a high-order, entropy stable, and well-balanced approximation for the Saint-Venant-Exner system.
Numerical tests are presented to verify the theoretical findings and demonstrate the performance and robustness of the proposed scheme.
\end{abstract}


\end{frontmatter}

\section{Introduction} 
% Intro ES-DGSEM for conservative systems
The discontinuous Galerkin method (DG) has become a common framework for the numerical approximation of hyperbolic conservation laws.
The method enjoys favorable dissipation and dispersion properties, geometric flexibility, and the possibility to create high-order approximations on an element-local stencil which provide good scalability.
However, the high-order nature responsible for many of these beneficial properties also makes the method prone to stability issues \cite{gassner2021novel}.
This becomes especially problematic in the presence of discontinuous and under-resolved simulations, where additional stabilization is often necessary to avoid spurious oscillations.
In the context of DG methods for nonlinear problems, the concept of entropy stable (ES) schemes \cite{tadmor2003entropy, merriam1989entropy}, where the solution is bounded by a convex mathematical entropy function, has recently received considerable attention.
These methods are \changeTwo{built} upon the idea of entropy conservative numerical fluxes introduced by Tadmor~\cite{tadmor1987numerical, tadmor2003entropy} for low-order finite volume (FV) schemes.
Carpenter, Fisher and collaborators \cite{fisher2013high, carpenter2014entropy} extended these entropy stable schemes to high-order finite difference methods using diagonal norm summation-by-parts (SBP) operators.
Subsequently, Gassner~\cite{gassner2013skew} applied this approach to nodal discontinuous Galerkin spectral element methods (DGSEM) with Legendre-Gauss-Lobatto (LGL) quadrature, as these methods satisfy the diagonal norm SBP property.
Since then, ES-DGSEM formulations have been successfully applied to a wide variety of problems, see e.g. \cite{ferrer2023, wintermeyer2017entropy, chan2022entropy, bohm2020entropy}.

% ES-DGSEM for nonconservative systems
While the application of entropy stable DGSEM has been actively researched for hyperbolic conservation laws, considerably less work has been done on hyperbolic systems that include nonconservative products.
The presence of nonconservative products introduces additional challenges when discontinuities are present, as they cannot be defined in the classical sense of distributions.
To resolve this, following the theory of Dal Maso, LeFloch and Murat~\cite{dal1995definition}, nonconservative products are defined by prescribing a family of paths through state space across a discontinuity.
From this theory, Parés~\cite{pares2006numerical} established the framework of path-conservative FV schemes formulated in terms of fluctuations that incorporate the path definition into the numerical scheme.
\changeThree{A problem in the application of path-conservative methods for problems involving discontinuities is that weak solutions depend on the chosen family of paths \cite{pares2006numerical}. Furthermore, even if the correct paths are known, path-conservative methods may fail to converge to the correct weak solution if numerical viscosity is inconsistent with the underlying physical viscosity mechanism \cite{castro2013entropy, abgrall2010comment}. Therefore, results obtained with path-conservative methods for problems involving strong shocks must be interpreted with caution.}  
In a similar way, nonconservative DG methods using the theory from \cite{dal1995definition} have been formulated in \cite{rhebergen2008discontinuos, franquet2012runge}.
Entropy stable formulations have been developed by Castro et al.~\cite{castro2013entropy}, by creating entropy conservative (EC) and ES fluctuations for the path-conservative FV schemes from \cite{pares2006numerical}.
Renac~\cite{renac2019entropy} extended the ES-DGSEM to nonconservative systems by introducing a special volume integral and high-order accurate EC fluctuations.
The method was further extended to balance laws in nonconservative form \cite{waruszewski2022entropy} and two-dimensional curvilinear meshes \cite{coquel2021entropy}.

% Open problems
Despite the recent progress in the development of entropy stable DG methods for nonconservative systems, several open issues remain. 
First, existing EC fluctuations typically assume a specific structure of the nonconservative product and it remains unclear when fluctuations that are both entropy and path-conservative exist.
As a consequence, formulating ES methods for new models may require substantial model-dependent derivations.
Second, for many applications the choice of dissipation term to construct ES fluctuations requires careful consideration to ensure structure-preserving properties of the scheme. 
Classical entropy symmetrization arguments due to Barth \cite{barth1999numerical} do not apply in the presence of nonconservative terms \cite{cordesse2019entropy}, which severely restricts the design of provably stable dissipation terms.

In this work, we address these issues by revisiting the nonconservative DGSEM formulation from \cite{renac2019entropy} and extend it with a more general class of EC fluctuations.
Using the structure of these new EC fluctuations, we provide a strategy that constructs EC fluctuations explicitly without requiring any model-specific derivations, and a second strategy, which provides more efficient fluctuations if a closed-form solution can be found.
We further introduce a novel blending procedure, that allows us to construct ES fluctuations from general matrix-valued dissipation terms.
Finally, we apply the resulting ES-DGSEM to the Saint-Venant-Exner (SVE) equations, which is a coupled system modeling the interaction of shallow water flow with sediment transport.
We consider the formulation derived by Fernández-Nieto et al.~\cite{fernandez2017formal}, which, due to the existence of a mathematical entropy function and the strong nonlinearity in the nonconservative product, makes for an ideal test case of the proposed strategies.

% Contributions & Novelties
\changeTwo{The main novelties in this work are the development of the general framework for constructing EC fluctuations, including both the general model-independent strategy and the efficient closed-form strategy, and the novel blending procedure to construct ES fluctuations from general matrix-valued dissipation terms. Furthermore, to the best of our knowledge, this work provides the first ES and high-order discretization for the SVE system.}

% Structure
In Section~\ref{sec:continuous_entropy_analysis}, we discuss the continuous entropy analysis and symmetrization properties for nonconservative systems.
We then formulate the nonconservative DGSEM and introduce strategies to construct EC and ES fluctuations in Section~\ref{sec:numerical_method}.
The proposed methodology is then applied in Section~\ref{sec:SVE_system} to develop a high-order, ES, and well-balanced approximation for the SVE system.
Numerical results for the SVE system are presented in Section~\ref{sec:results} to validate the theoretical findings and demonstrate convergence, well-balancedness, and ES properties.
Finally, we conclude our findings in Section~\ref{sec:conclusion}.

\section{Entropy analysis}\label{sec:continuous_entropy_analysis}
	We consider nonlinear systems with nonconservative terms of the form
	\begin{equation}
		\statevec{u}_t + \statevec{f}(\statevec{u})_x + \statemat{B}(\statevec{u})\statevec{u}_x = 0,
		\label{eq:nonconservative_system}
	\end{equation}
	where $\statevec{f}(\statevec{u}) \in \mathbb{R}^{\statedim}$ denotes a conservative flux function, $\statemat{B}(\statevec{u})\statevec{u}_x$ \changeTwo{is} a nonconservative product with coefficient matrix $\statemat{B}(\statevec{u}) \in \mathbb{R}^{\statedim \times \statedim}$ and $\statevec{u}(x, t) \in \stateset$ is the vector of solution variables. Herein, we assume this nonconservative product contains only true nonconservative terms, which cannot be expressed as the gradient of a vector-valued function. 
	% A necessary criterium for this is that the mixed second derivatives of A are symmetric. 
	Systems such as \eqref{eq:nonconservative_system} can also be represented in quasilinear form
	\begin{equation}
		\statevec{u}_t + \statemat{A}(\statevec{u})\statevec{u}_x = 0,
		\label{eq:system_quasilinear_form}
	\end{equation}
	with the generalized Jacobian $\statemat{A} \coloneqq \statevec{f}_\statevec{u} + \statemat{B}$, composed of the flux Jacobian $\statevec{f}_{\statevec{u}}(\statevec{u}) \in \mathbb{R}^{n \times n}$ and the nonconservative coefficient matrix $\statemat{B}$. 
	
	In the following, we give a brief overview to the continuous entropy theory for nonlinear systems and then discuss some key differences that arise in the presence of nonconservative terms.
	A nonconservative system \eqref{eq:nonconservative_system} may admit an entropy pair consisting of a convex mathematical entropy function $S(\statevec{u}) \in \mathbb{R}$ and a corresponding entropy flux $q(\statevec{u}) \in \mathbb{R}$. 
	Such entropy pairs must satisfy the entropy flux compatibility condition
	\begin{equation}
		q_{\statevec{u}} = S_{\!\statevec{u}}^{\!T}(\statevec{f}_{\statevec{u}} + \statemat{B}) = \statevec{w}^T(\statevec{f}_{\statevec{u}} + \statemat{B}),
		\label{eq:entropy_flux_compatibility_relation}
	\end{equation}
	where $\statevec{w}(\statevec{u}) \coloneqq S_{\!\statevec{u}}(\statevec{u}) \in \mathbb{R}^n$ denotes a new set of entropy variables for the system. 
	For a strictly convex entropy function the Hessian $S_{\!\statevec{u}\statevec{u}} = \statevec{w}_\statevec{u}$ is symmetric positive definite and, thus, provides a one-to-one mapping between solution variables $\statevec{u}$ and entropy variables $\statevec{w}$.
	
	Systems equipped with an entropy pair satisfy an additional conservation law for the mathematical entropy. For smooth solutions, multiplying system \eqref{eq:nonconservative_system} from the left with entropy variables and integrating over a domain $\Omega$ yields an integral version of this auxiliary scalar entropy conservation law
	\begin{equation}
	 	\int_{\Omega} \statevec{w}^T(\statevec{u}_t + \statevec{f}_x + \statemat{B}\statevec{u}_x) \, \text{d}x
	 	= 	
	 	\int_{\Omega} S_{\!t} + q_x \, \text{d}x 
	 	= 
	 	0.
	 	\label{eq:entropy_conservation_law}
	\end{equation}
	In the presence of discontinuities physically relevant solutions of \eqref{eq:nonconservative_system} should instead satisfy an entropy inequality
	\begin{equation}
		\int_{\Omega} S_{\!t} + q_x \, \text{d}x \leq 0.
		\label{eq:entropy_inequality}
	\end{equation}
	
	For hyperbolic conservative systems, the existence of a mathematical entropy function is directly linked to symmetrization and skew-symmetric forms.
	In \cite{godunov1961interesting, mock1980systems} it was shown that a conservative hyperbolic system is symmetrized by a change of variables if and only if there exists an entropy pair for the system.
	These results establish that conservative hyperbolic systems ($\statemat{B} = \statemat{0}$) are symmetrized when written in terms of entropy variables
	\begin{equation}
		\statevec{u}_t + \statevec{f}_x = \statemat{H}\statevec{w}_t + \statevec{f}_{\statevec{u}} \statemat{H} \statevec{w}_x = 0,
		\label{eq:cons_system_symmetric_form}
	\end{equation}
	where we introduce the Hessian of the entropy function $\statemat{H}^{-1}(\statevec{u}) \coloneqq \statevec{w}_{\statevec{u}}$ and its inverse $\statemat{H}(\statevec{w}(\statevec{u})) \coloneqq \statevec{u}_{\statevec{w}}$, which provide the mapping between both variable sets.
	% Nonconservative systems
	While the theorems from \cite{mock1980systems, godunov1961interesting} establish this symmetry connection for conservative systems, these results do not extend in the presence of nonconservative products.
	In fact it was demonstrated in \cite{cordesse2019entropy} that, while nonconservative systems may still satisfy an auxiliary entropy conservation law, the connection between entropy functions and symmetrization breaks down.
	
	To examine the symmetrization properties and its breakdown, we differentiate the compatibility condition \eqref{eq:entropy_flux_compatibility_relation} and then apply the chain rule to obtain 
	\begin{equation}
		\statemat{H}^{-1} (\statevec{f}_\statevec{u} + \statemat{B}) + \statevec{w}^T(\statevec{f}_{\statevec{u}\statevec{u}} + \statemat{B}_\statevec{u}) = q_{\statevec{u}\statevec{u}}.
	\end{equation}
	In this expression, we identify ${q}_{\statevec{u}\statevec{u}}$ to be symmetric as it is the Hessian of the entropy flux, while $\statevec{w}^T\statevec{f}_{\statevec{u}\statevec{u}}$ is symmetric, as it involves a linear combination of Hessians $\statevec{f}_{\statevec{u}\statevec{u}} \in \mathbb{R}^{n \times n \times n} = \left((f_1)_{\statevec{u}\statevec{u}}, ..., (f_n)_{\statevec{u}\statevec{u}}\right)$. 
	% See e.g. https://en.wikipedia.org/wiki/Hessian_matrix#Vector-valued_functions
	Grouping the known symmetric terms on the right, we find that
	\begin{equation}
		\statemat{H}^{-1}(\statevec{f}_{\statevec{u}} + \statemat{B}) + \statevec{w}^T(\statemat{B}_\statevec{u}) \quad \text{is symmetric}.
	\end{equation}
	For conservative systems ($\statemat{B}=\statemat{0}$), this yields the known result that $\statemat{H}^{-1}$ acts as left symmetrizer for the flux Jacobian $\statevec{f}_{\statevec{u}}$.
	With $\statemat{H}^{-1}$ symmetric positive definite, this is equivalent to $\statemat{H}$ being a right symmetrizer to $\statevec{f}_{\statevec{u}}$, establishing the symmetry of \eqref{eq:cons_system_symmetric_form}, see Lemma \ref{lemma:left_right_symmetrization} in~\ref{app:lemma4}.
	
	In contrast, for nonconservative systems ($\statemat{B} \neq \statemat{0}$) only the sum $\statemat{H}^{-1}\statevec{f}_{\statevec{u}} + \statevec{w}^TB_u$ is symmetric, hence neither $\statemat{H}^{-1}(\statevec{f}_{\statevec{u}} + B)$ nor $(\statevec{f}_{\statevec{u}} + B)\statemat{H}$ remains symmetric.
	Therefore, formulating the nonconservative system \eqref{eq:system_quasilinear_form} in terms of entropy variables
    \begin{equation}
    	\statevec{u}_t + \statevec{f}_x + \statemat{B}\statevec{u}_x = \statemat{H}\statevec{w}_t + (\statevec{f}_{\statevec{u}} + \statemat{B}) \statemat{H} \statevec{w}_x = 0,
    	\label{eq:noncons_system_symmetric_form}
    \end{equation}
    does not yield a symmetrized form of the system.
    
    While this does not affect the existence of an entropy pair or the corresponding entropy conservation law for nonconservative systems, the loss of symmetry restricts the design of entropy stable dissipation operators for numerical approximations.
   	In particular, it prevents the use of the eigenvector scaling of Barth \cite{barth1999numerical} commonly used to create matrix dissipation operators for conservative systems as it relies on entropy symmetrization.
	
\section{Entropy stable DGSEM}\label{sec:numerical_method}
	In the following we introduce a DGSEM for the one-dimensional nonconservative system \eqref{eq:system_quasilinear_form}.
	Similar to \cite{renac2019entropy}, we consider a special volume integral and provide a novel ansatz for EC fluctuations to obtain an EC discretization.
	\changeTwo{We then provide two strategies how to construct such EC fluctuation with the novel ansatz. The first strategy in Section~\ref{sec:ec_strategy_1} enables a general construction of EC fluctuations without any model dependent derivations, while the second strategy in Section~\ref{sec:ec_strategy_2} is more efficient, but requires model dependent derivations.}
	Finally, we propose a blending procedure to construct ES fluctuations, that are introduced at element interfaces to construct a scheme that recovers a semidiscrete version of the entropy inequality \eqref{eq:entropy_inequality}.
	
	\subsection{Nonconservative DGSEM}
	First, we divide the physical domain $\Omega$ into $K$ non-overlapping elements $\Omega^k = [x_{k-1}, x_k]$ of size $\Delta x_k = x_k - x_{k-1}$\changeThree{with $k=1\ldots K$.}
	% Introduce the DG approximation space
	On each element $\Omega^k$, we introduce a local vector-valued approximation space $[\mathbb{P}^{\polydeg}(\Omega^k)]^{\statedim}$ of polynomials up to degree $\polydeg$ that are combined to obtain the DG approximation space
	\begin{equation}
		\mathbb{V}^{\polydeg} \coloneqq \{v: v|_{\Omega^k} \in \mathbb{P}^{\polydeg}, \; \forall \Omega^k \in \Omega\},
	\end{equation}
	with possibly discontinuous functions across element interfaces. The numerical approximation is then denoted by $\statevec{u} \in [\mathbb{V}^{\polydeg}]^{\statedim}$.
	Furthermore, we denote $\statevec{u}_k^{\pm} := \lim\limits_{\epsilon \rightarrow 0} \statevec{u}(x_k \pm \epsilon)$
	as the one-sided limits of the discontinuous solution $\statevec{u}$ at the interface node $x_k$ from the left $\boldsymbol{-}$ and right $\boldsymbol{+}$, respectively. \changeThree{Thus, on the element $\Omega^k$ the left interface node, $x_{k-1}$, is denoted with $\boldsymbol{+}$ and the right interface node, $x_k$, is denoted with $\boldsymbol{-}$}. We multiply with a test function $\phi \in \mathbb{V}^{\polydeg}$ and integrate over the domain to obtain the element-wise semidiscrete weak form  of \eqref{eq:system_quasilinear_form} \cite{rhebergen2008discontinuos, renac2019entropy}
\begin{equation}
	\int_{\Omega^k} \phi \statevec{u}_t \text{d}x
	+
	\int_{\Omega^k} \phi \statemat{A}\statevec{u}_x \text{d}x +
	\phi_{k-1}^+ \statevec{D}^+(\statevec{u}^-_{k-1}, \statevec{u}^+_{k-1}) +
	\phi_{k}^-\statevec{D}^-(\statevec{u}^-_{k}, \statevec{u}^+_{k}) = 0, \quad \forall \phi \in \mathbb{V}^{\polydeg}.
	\label{eq:weak_form_fluctuation}
\end{equation}
	Following the idea of path-conservative schemes \cite{pares2006numerical}, the integrals at interfaces in  \eqref{eq:weak_form_fluctuation} are split into a pair of fluctuations $\statevec{D}^{\pm}:\stateset \times \stateset \rightarrow \mathbb{R}^n$ \changeThree{that represent the contributions to elements adjacent to the corresponding interface. These fluctuations} are Lipschitz continuous functions that satisfy the following consistency and path-conservation condition for two states $\statevec{u}_L, \statevec{u}_R \in \mathbb{R}^n$
	\begin{align}
		\statevec{D}^{\pm}(\statevec{u},\statevec{u}) 
		&= 0,\label{eq:fluctuation_condition_a}\\
		\statevec{D}^-(\statevec{u}_L, \statevec{u}_R) 
		+
		\statevec{D}^+(\statevec{u}_L, \statevec{u}_R)
		&= 
		\int_0^1 \statemat{A}(\Phi(s;\statevec{u}_L, \statevec{u}_R)) \partial_{\!s} \Phi(s;\statevec{u}_L, \statevec{u}_R) \text{d}s\label{eq:path_conservation_condition}.
	\end{align}
	This formulation assumes the definition of the nonconservative product as a Borel measure as introduced in \cite{dal1995definition}, where the nonconservative product at discontinuities depends on the Lipschitz continuous path function $\Phi(s;\statevec{u}_L, \statevec{u}_R) : [0,1] \times \stateset \times \stateset \rightarrow \stateset$, which connects the two states $\statevec{u}_L$, $\statevec{u}_R$ and satisfies the consistency conditions
	\begin{equation}
		\begin{aligned}
			\ncpath{0}{\statevec{u}_L}{\statevec{u}_R} &= \statevec{u}_L, \quad 
			\ncpath{1}{\statevec{u}_L}{\statevec{u}_R} = \statevec{u}_R, \quad
			\ncpath{s}{\statevec{u}}{\statevec{u}} &= \statevec{u}.
		\end{aligned}
	\end{equation}
	\changeThree{The fluctuations contain the contribution of the nonconserative product at interfaces and play the role of incorporating a specified path function $\Phi(s;\statevec{u}_L, \statevec{u}_R)$ into the numerical method.}
	
	Once the path and fluctuations are chosen, each element $\Omega^k$ is mapped onto the reference element $E=[-1,1]$ with the mapping
	\begin{equation}
		x 
		= X^{k}(\xi) 
		= x_{k-1} + \frac{\xi + 1}{2}\Delta x_k, \quad \xi \in [-1, 1].
		\label{eq:element_mapping_func}
	\end{equation}
	We apply the affine map \eqref{eq:element_mapping_func} and rewrite the spatial derivatives in \eqref{eq:weak_form_fluctuation} in terms of $\xi$, to obtain the weak formulation on the reference element
	\begin{equation}
		\frac{\Delta x_k}{2}\int_{E} \phi\statevec{u}_t \text{d}x
		+
		\int_{E} \phi\statemat{A}\statevec{u}_{\xi} \text{d}\xi +
		\phi_{k-1}^+ \statevec{D}^+(\statevec{u}^-_{k-1}, \statevec{u}^+_{k-1}) +
		\phi_{k}^-\statevec{D}^-(\statevec{u}^-_{k}, \statevec{u}^+_{k}) = 0, \quad \forall \phi \in \mathbb{V}^{\polydeg}.
		\label{eq:weak_form_reference_element}
	\end{equation}
	On the reference element $E$, we approximate the solution with a local basis of polynomials of degree $N$
	\begin{equation}
		\statevec{u} \approx \statevec{U}^k(\xi) = \sum\limits_{j=0}^N \statevec{U}_{j}^k l_j(\xi)
		\label{eq:definition_polynomial_basis}
	\end{equation}
	that is spanned by the nodal Lagrange basis functions $l_j(\xi)$ with interpolation nodes $\{\xi\}_{i=0}^N$ set at Legendre-Gauss-Lobatto (LGL) points.
	For this choice of basis function we further define the discrete derivative operator \cite{kopriva2009implementing}
	\begin{equation}
		\mathcal{D}_{ij} := \frac{\text{d}l_j(\xi)}{\text{d} \xi}\Big|_{\xi = \xi_i}, \quad i,j=0,...,N,
		\label{eq:definition_derivative_matrix}
	\end{equation}
	approximate all integrals with LGL quadrature rules, and collocate interpolation and quadrature nodes
	\begin{equation}
		\int_{-1}^1 f(\xi) d\xi \approx \sum\limits_{j=0}^N \omega_j f(\xi_j),
		\label{eq:definition_gauss_quadrature}
	\end{equation}
	with LGL quadrature weights $\{\omega_j\}_{j=0}^N$. 
	As shown in \cite{gassner2013skew}, this specific choice of operators satisfies the diagonal norm SBP property
	\begin{equation}
		\mathcal{Q} + \mathcal{Q}^T = \mathcal{B}
		\label{eq:sbp_property}
	\end{equation}
	with the undivided difference matrix $\mathcal{Q}_{ij} = \omega_i \mathcal{D}_{ij}$  and boundary matrix $\mathcal{B}_{ij} = \delta_{iN}\delta_{jN} - \delta_{i0}\delta_{j0}$, where $\delta$ denotes the Kronecker delta.
	This property represents a discrete analog to integration-by-parts, which is key to prove entropy stability.
	
	We use the polynomial approximation \eqref{eq:definition_polynomial_basis} for the solution variables, set the Lagrange basis polynomials as test functions, apply LGL quadrature \eqref{eq:definition_gauss_quadrature}, and replace derivatives with the differentiation matrix \eqref{eq:definition_derivative_matrix} in \eqref{eq:weak_form_reference_element} to obtain the nonconservative DGSEM formulation
	\begin{equation}
		\omega_i \frac{\Delta x_k}{2} \dot{\statevec{U}}_i^k + 
		\omega_i \sum\limits_{m=0}^N \mathcal{D}_{im} \statemat{A}(\statevec{U}_i^k)\statevec{U}_m^k
		+ \delta_{i0}\statevec{D}^+(\statevec{U}_N^{k-1}, \statevec{U}_0^k)
		+ \delta_{iN}\statevec{D}^-(\statevec{U}_N^{k}, \statevec{U}_0^{k+1}) = 0.
		\label{eq:nc_DGSEM}
	\end{equation}
	
\subsection{Entropy conservative DGSEM}
	To obtain an approximation that recovers the entropy conservation law in the semidiscrete case we replace the volume integral in the DGSEM formulation \eqref{eq:nc_DGSEM} with a nonconservative version of the flux differencing volume integral from \cite{gassner2016split, carpenter2014entropy}
	\begin{equation}
		\omega_i \sum\limits_{m=0}^N \mathcal{D}_{im} \statemat{A}(\statevec{U}_i^k)\statevec{U}_m^k
		\approx
		\omega_i \sum\limits_{m=0}^N 2\mathcal{D}_{im} \statevec{D}^-(\statevec{U}_i^k,\statevec{U}_m^k),
		\label{eq:flux_diff_volume_integral}
	\end{equation}
	so that the DGSEM \eqref{eq:nc_DGSEM} is written in flux differencing form
	\begin{equation}
			\omega_i \frac{\Delta x_k}{2} \dot{\statevec{U}}_i^k + 
			\omega_i \sum\limits_{m=0}^N 2\mathcal{D}_{im} \statevec{D}^-(\statevec{U}_i^k,\statevec{U}_m^k)
			+ \delta_{i0}\statevec{D}^+(\statevec{U}_N^{k-1}, \statevec{U}_0^k)
			+ \delta_{iN}\statevec{D}^-(\statevec{U}_N^{k}. \statevec{U}_0^{k+1}) = 0.
		\label{eq:flux_diff_DGSEM}
	\end{equation}
	This formulation provides a consistent and high-order accurate approximation, provided the fluctuations satisfy the consistency condition
	\begin{equation}
			\frac{\partial \statevec{D}^-(\statevec{u}_L, \statevec{u}_R)}{\partial \statevec{u}_R}\Big|_{\statevec{u}_R = \statevec{u}_L}
            =
			\frac{1}{2}\statemat{A}(\statevec{u}_L).
			\label{eq:fluctuation_condition_c}
	\end{equation}
	While this approach is motivated by the work from \cite{renac2019entropy}, the volume integral differs.
	\changeOne{In this work, we adopt a different structure for the volume fluctuations together with a new consistency condition to obtain an alternative way to construct EC volume fluctuations.
	Specifically, the novel ansatz was chosen in such a way that it enables the construction of EC volume fluctuations without any need for model dependent derivations as well as a second strategy to obtain more efficient, but model dependent closed-form EC volume fluctuations, as discussed in Sections~\ref{sec:ec_strategy_1} and \ref{sec:ec_strategy_2}.}
	
	\begin{lemma}\label{lemma:accuracy_dgsem}
		The flux differencing volume integral \eqref{eq:flux_diff_volume_integral} with fluctuations satisfying \eqref{eq:fluctuation_condition_c} approximates the nonconservative product $\statemat{A}\statevec{u}_{\xi}$ with the same order of accuracy $p$ as the derivative matrix \eqref{eq:definition_derivative_matrix}.
	\end{lemma}
	\begin{proof}
		 The proof for accuracy uses that $\mathcal{D}$ is a $p$th order accurate derivative operator. We use the chain rule and apply condition \eqref{eq:fluctuation_condition_c} to obtain 
		\begin{equation}
			\begin{aligned}
				\sum\limits_{m=0}^N 2 \mathcal{D}_{im} \statevec{D}^-(\statevec{u}(\xi_i), \statevec{u}(\xi_m)) 
				&= 
				2 \frac{\partial \statevec{D}^-(\statevec{u}(\xi_i), \statevec{u}(\xi_m))}{\partial \xi_m}\Big|_{\xi_m = 	\xi_i} + \mathcal{O}(\Delta \xi^p)
				\\&=
				2 \frac{\partial \statevec{D}^-(\statevec{u}_i, \statevec{u}_m)}{\partial \statevec{u}_m}\Big|_{u_m 	= u_i} \frac{\partial \statevec{u}(\xi)}{\partial \xi}\Big|_{\xi = \xi_i} + \mathcal{O}(\Delta \xi^p)
				\\&=
				\statemat{A}(\statevec{u}(\xi)) \frac{\partial \statevec{u}(\xi)}{\partial \xi}\Big|_{\xi = \xi_i} 
				+
				\mathcal{O}(\Delta \xi^p)
			\end{aligned}
		\end{equation}\qed
	\end{proof}
	
	Next, we show that the flux differencing formulation creates a scheme that preserves an integral version of the entropy conservation law \eqref{eq:entropy_conservation_law} in the semidiscrete case.
	To do so, we assume the following structure on the fluctuations
	\begin{equation}
	   \statevec{D}^{-}(\statevec{u}_L, \statevec{u}_R) = - \statevec{D}^{+}(\statevec{u}_R, \statevec{u}_L)\label{eq:fluctuation_condition_b}
    \end{equation}
    and that fluctuations satisfy a discrete version of the entropy conservation condition
	\begin{equation}
		q(\statevec{u}_R) - q(\statevec{u}_L) = \statevec{w}(\statevec{u}_L)^T\statevec{D}^-(\statevec{u}_L, \statevec{u}_R) + \statevec{w}(\statevec{u}_R)^T\statevec{D}^+(\statevec{u}_L, \statevec{u}_R).
		\label{eq:entropy_conservation_condition}
	\end{equation}
	This leads to the definition of entropy conservative fluctuations.
	\begin{definition}[EC fluctuation]\label{def:ec_fluctuation}
		A fluctuation is defined as entropy conservative $\statevec{D}_{EC}^{\pm}$ if it satisfies the conditions \eqref{eq:fluctuation_condition_a}, \eqref{eq:fluctuation_condition_b}, \eqref{eq:fluctuation_condition_c}, \eqref{eq:path_conservation_condition}, \eqref{eq:entropy_conservation_condition}, collected below for convenience
		\begin{subequations}\label{eq:ec_conditions}
		\begin{align}
			\statevec{D}^{\pm}(\statevec{u},\statevec{u}) 
			&= 0,\label{eq:ec_condition_1}
			\\
			\statevec{D}^-(\statevec{u}_L, \statevec{u}_R) 
			+
			\statevec{D}^+(\statevec{u}_L, \statevec{u}_R),
			&= 
			\int_0^1 \statemat{A}(\Phi(s;\statevec{u}_L, \statevec{u}_R)) \partial_{\!s} \Phi(s;\statevec{u}_L, \statevec{u}_R),\label{eq:ec_condition_2}
			\\
			\statevec{D}^{-}(\statevec{u}_L, \statevec{u}_R) + \statevec{D}^{+}(\statevec{u}_R, \statevec{u}_L)
			&= 0,\label{eq:ec_condition_3}
			\\
			\statevec{w}_L^T\statevec{D}^-(\statevec{u}_L, \statevec{u}_R) + \statevec{w}_R^T\statevec{D}^+(\statevec{u}_L, \statevec{u}_R) &= q_R - q_L,\label{eq:ec_condition_4}
			\\
			\frac{\partial \statevec{D}^-(\statevec{u}_L, \statevec{u}_R)}{\partial \statevec{u}_R}\Big|_{\statevec{u}_R = \statevec{u}_L}
			&= \frac{1}{2}\statemat{A}(\statevec{u}_L).\label{eq:ec_condition_5}
		\end{align}
		\end{subequations}
	\end{definition}
	\begin{lemma}\label{lemma:ec_dgsem}
		The nonconservative DGSEM with flux differencing formulation \eqref{eq:flux_diff_DGSEM}
		and entropy conservative fluctuations $\statevec{D}^{\pm}_{EC}$ as defined in Definition~\ref{def:ec_fluctuation} recovers the semi-discrete entropy conservation law at the element level
		\begin{equation}
			\sum\limits_{i=0}^N \omega_i\frac{\Delta x_k}{2} \frac{\partial S(\statevec{U}_i^k)}{\partial t} 
			=
			Q(\statevec{U}_N^{k-1}, \statevec{U}_0^k)
			-
			Q(\statevec{U}_N^k, \statevec{U}_0^{k+1}),
		\end{equation}
		where $Q$ denotes a numerical entropy flux, consistent with the physical entropy flux $Q(\statevec{u}, \statevec{u}) = q(\statevec{u})$.
	\end{lemma}
	\begin{proof} % Entropy stable volume integral
		To demonstrate entropy conservation for the flux differencing formulation \eqref{eq:flux_diff_DGSEM}, we multiply with entropy variables $\statevec{W}_i^k := w(\statevec{U}_i^k)$ and sum over the nodal values to obtain
		\begin{equation}
			\begin{aligned}
				\sum\limits_{i=0}^N \omega_i\frac{\Delta x_k}{2} {\statevec{W}_i^k}^T\frac{\partial \statevec{U}_i^k}{\partial t} = 
				- \sum\limits_{i=0}^N {\statevec{W}_i^k}^T \sum\limits_{m=0}^N 2\mathcal{D}_{im} \statevec{D}^-_{EC}(\statevec{U}_i^k,\statevec{U}^k_m)
				- \sum\limits_{i=0}^N \delta_{i0}{\statevec{W}_i^k}^T\statevec{D}^+_{EC}(\statevec{U}_N^{k-1}, \statevec{U}_0^k)
				- \sum\limits_{i=0}^N \delta_{iN}{\statevec{W}_i^k}^T\statevec{D}^-_{EC}(\statevec{U}_N^{k}, \statevec{U}_0^{k+1}).
			\end{aligned}
		\end{equation}
		Next, we assume that the chain rule holds in time to obtain an evolution equation for the entropy
		\begin{equation}
			\begin{aligned}
				\sum\limits_{i=0}^N \omega_i\frac{\Delta x_k}{2}\frac{\partial S(\statevec{U}_i^k)}{\partial t} = 
				- \sum\limits_{i=0}^N {\statevec{W}_i^k}^T \sum\limits_{m=0}^N 2\mathcal{D}_{im} \statevec{D}^-_{EC}(\statevec{U}_i^k,\statevec{U}_m^k)
				- \statevec{W}_0^T\statevec{D}^+_{EC}(\statevec{U}_N^{k-1}, \statevec{U}_0^k)
				- \statevec{W}_N^T\statevec{D}^-_{EC}(\statevec{U}_N^{k}, \statevec{U}_0^{k+1})
			\end{aligned}.
			\label{eq:semidiscrete_ec_law_step_1}
		\end{equation}
		
		Next, we examine the volume term in \eqref{eq:semidiscrete_ec_law_step_1} to demonstrate that volume contributions can be rewritten as entropy fluxes at interfaces.
		We first apply the SBP property \eqref{eq:sbp_property} and use that the boundary contribution $\mathcal{B}_{im}\statevec{D}^-_{EC}(\statevec{U}_i^k, \statevec{U}_m^k)$ vanishes due to $\eqref{eq:ec_condition_1}$.
		Then, we first swap indices $i \leftrightarrow m$, apply the condition \eqref{eq:ec_condition_3} and finally use the entropy conservation condition \eqref{eq:ec_condition_4} and the SBP property \eqref{eq:sbp_property} to obtain the result that the volume integral generates the entropy fluxes at the boundary  
		\begin{equation}
			\begin{aligned}
				\sum\limits_{i = 0}^N \omega_i {\statevec{W}_i^k}^T \sum\limits_{m=0}^N 2\mathcal{D}_{im} \statevec{D}^-_{EC}(\statevec{U}_i^k, \statevec{U}_m^k) 
				&=
				\sum\limits_{i,m = 0}^N {\statevec{W}_i^k}^T 2\mathcal{Q}_{im} \statevec{D}^-_{EC}(\statevec{U}_i^k, \statevec{U}_m^k) 
				\\&=
				\sum\limits_{i,m = 0}^N {\statevec{W}_i^k}^T (\mathcal{Q}_{im} - \mathcal{Q}_{mi} + \mathcal{B}_{im}) \statevec{D}^-_{EC}(\statevec{U}_i^k, \statevec{U}_m^k) 
				\\&=
				\sum\limits_{i,m = 0}^N \mathcal{Q}_{im}\left( {\statevec{W}_i^k}^T  \statevec{D}^-_{EC}(\statevec{U}_i^k, \statevec{U}_m^k) 
				- 
				\statevec{W}_m^T \statevec{D}^{\changeOther{-}}_{EC}(\statevec{U}_m, \statevec{U}_i)\right)
				\\&=
				\sum\limits_{i,m = 0}^N \mathcal{Q}_{im} \left({\statevec{W}_i^k}^T \statevec{D}^-_{EC}(\statevec{U}_i^k, \statevec{U}_m^k)
				+ 
				\statevec{W}_m^T \statevec{D}^{\changeOther{+}}_{EC}(\statevec{U}_i^k, \statevec{U}_m^k)\right)
				\\&=
				\sum\limits_{i,m = 0}^N \mathcal{Q}_{im}(q_m - q_i) = q_N - q_0.
			\end{aligned}
		\end{equation}
		Inserting this result in \eqref{eq:semidiscrete_ec_law_step_1}, we then recover the element-wise form of the semi-discrete entropy conservation law
		\begin{equation}
			\begin{aligned}
				\sum\limits_{i=0}^N \omega_i\frac{\Delta x_k}{2} \frac{\partial S(\statevec{U}_i^k)}{\partial t} 
				&= 
				\left(q_0^k 
				- {\statevec{W}_0^k}^T\statevec{D}^+_{EC}(\statevec{U}_N^{k-1}, \statevec{U}_0^k)\right)
				-\left(q_N^k 
				+ {\statevec{W}_N^k}^T\statevec{D}^-_{EC}(\statevec{U}_N^{k}, \statevec{U}_0^{k+1})\right)
				\\
				&=
				Q(\statevec{U}_N^{k-1}, \statevec{U}_0^k)
				-Q(\statevec{U}_N^k, \statevec{U}_0^{k+1}),
			\end{aligned}
		\end{equation}
		where using \eqref{eq:ec_condition_4}, we introduce the numerical entropy fluxes
		% Note: This follows directly from the EC condition.
		\begin{equation}
			Q(\statevec{U}_N^{k-1}, \statevec{U}_0^k)
			=
			\left(q_0^k - {\statevec{W}_0^k}^T\statevec{D}^+_{EC}(\statevec{U}_N^{k-1}, \statevec{U}_0^k)\right)
			=
			\left(q_N^{k-1} + {\statevec{W}_N^{k-1}}^T\statevec{D}^-_{EC}(\statevec{U}_N^{k-1}, \statevec{U}_0^{k})\right).
		\end{equation}
		\qed
	\end{proof}
	
	\begin{theorem}
		The nonconservative DGSEM with flux differencing formulation
		\begin{equation*}
			\omega_i \frac{\Delta x_k}{2} \frac{\partial \statevec{U}_i^k}{\partial t} 
			+
			\omega_i \sum\limits_{m=0}^N 2\mathcal{D}_{im} \statevec{D}^-_{EC}(\statevec{U}_i^k,\statevec{U}_m^k)
			+ \delta_{i0}\statevec{D}_{EC}^+(\statevec{U}_N^{k-1}, \statevec{U}_0^k)
			+ \delta_{iN}\statevec{D}^-_{EC}(\statevec{U}_N^{k}, \statevec{U}_0^{k+1}) = 0,
		\end{equation*}
		and EC fluctuations $\statevec{D}^{\pm}_{EC}$ as defined in Definition~\ref{def:ec_fluctuation}
		provides a design order accurate and entropy conservative approximation of system \eqref{eq:system_quasilinear_form}.
	\end{theorem}
	\begin{proof}
		The proof follows directly from Lemma \ref{lemma:accuracy_dgsem} and Lemma \ref{lemma:ec_dgsem}.
		\qed
	\end{proof}
	With the conditions for high-order accuracy and entropy conservation established, it remains to demonstrate how to construct the EC fluctuations.
    We provide two strategies to construct such entropy conservative fluctuations.
	In the first strategy, fluctuations are constructed explicitly via a linear path in entropy variables, while in the second strategy, we prescribe a structure on the fluctuations that guarantees the high-order accuracy, and then derive an instance that also complies with the entropy conservation condition \eqref{eq:ec_condition_4}.
	\changeTwo{The second strategy should generally be preferred, due to its favorable efficiency.
	However, deriving closed-form expressions for the fluctuations can be cumbersome and their existence for general models is not guaranteed.
	In such cases, the first strategy provides a general and straightforward alternative, at the cost of increased computational effort.}
	
	\subsubsection{Entropy conservative fluctuations: Strategy 1}\label{sec:ec_strategy_1}
	The first strategy can be viewed as a generalization of Tadmor's entropy conservative fluxes \cite{tadmor1987numerical} to nonconservative systems.
	The approach builds on the method proposed by Castro et al. \cite{castro2013entropy} to construct entropy conservative fluctuations by integrating the entropy flux compatibility condition \eqref{eq:entropy_flux_compatibility_relation}. In \cite{castro2013entropy} this strategy was used to construct entropy conservative fluctuations for a general family of paths.
	Here, we modify the approach by writing the compatibility condition in terms of entropy variables to obtain $q_w = \statevec{w}^T\statemat{A}\,\statemat{H}$, rather than solution variables. Furthermore, we consider the linear path in entropy variables
	\begin{equation}
		\Phi(s):=\Phi(s;\statevec{u}_L,\statevec{u}_R) = \statevec{w}_L + s \jump{\statevec{w}}.
		\label{eq:path_entropy_vars}
	\end{equation}
	We introduce notation to denote the jump and arithmetic average between two states $(\cdot)_L$ and $(\cdot)_R$ as
	\begin{equation}
		 \jump{\cdot} \coloneqq (\cdot)_R - (\cdot)_L,
		 \quad
		 \avg{\cdot} \coloneqq \frac{1}{2}\left((\cdot)_R + (\cdot)_L\right).
    \end{equation}
    
	Integrating the entropy flux compatibility condition and introducing the linear path \eqref{eq:path_entropy_vars} yields
	\begin{equation}
		\jump{q} = \int_0^1 \Phi(s)^T\statemat{A}(\statevec{u}(\Phi(s)))\statemat{H}(\Phi(s)) \frac{\partial \Phi(s)}{\partial s} \text{d}s
		=
		\int_0^1 (\statevec{w}_L + s \jump{\statevec{w}})^T\statemat{A}\,\statemat{H}(\Phi(s)) \jump{\statevec{w}} \text{d}s.
	\end{equation}
	The resulting integral is split into distinct contributions, by factoring out $\statevec{w}_L$ and $\statevec{w}_R$ to recover the entropy conservation condition \eqref{eq:ec_condition_4}, which defines the entropy conservative fluctuations
	\begin{equation}
		\begin{aligned}
			\jump{q}
			&=
			(\statevec{w}_L)^T \int_0^1 (1 - s)\statemat{A}\,\statemat{H}(\Phi(s)) \jump{\statevec{w}} \text{d}s
			+
			(\statevec{w}_R)^T \int_0^1 s\statemat{A}\,\statemat{H}(\Phi(s)) \jump{\statevec{w}} \text{d}s
			\\&=
			(\statevec{w}_L)^T \statevec{D}^-\left(\statevec{u}_L, \statevec{u}_R\right)
			+
			(\statevec{w}_R)^T \statevec{D}^+\left(\statevec{u}_L, \statevec{u}_R\right).
			\label{eq:entropy_conservation_condition_strategy_1}
		\end{aligned}
	\end{equation}
	
	\begin{theorem}\label{thm:ec_fluctuation_strategy_1}
		The fluctuations 
		\begin{equation}
			\statevec{D}^-(\statevec{u}_L, \statevec{u}_R)
			=
			\int_0^1 \left(1-s\right)\statemat{A}\,\statemat{H}(\Phi(s)) \, \text{d}s \jump{\statevec{w}},
			\quad
			\statevec{D}^+(\statevec{u}_L, \statevec{u}_R)
			=
			\int_0^1 s\statemat{A}\,\statemat{H}(\Phi(s)) \, \text{d}s \jump{\statevec{w}}
			\label{eq:definition_ec_flux_strategy_1}
		\end{equation}
		evaluated along the linear path in entropy variables \eqref{eq:path_entropy_vars} are entropy conservative fluctuations according to Definition~\ref{def:ec_fluctuation}.
	\end{theorem}
	\begin{proof}
		That the path-conservation condition \eqref{eq:ec_condition_2}, entropy conservation condition \eqref{eq:ec_condition_4} and \eqref{eq:ec_condition_1} hold follows directly from the construction.
		To demonstrate \eqref{eq:ec_condition_5} we evaluate the Jacobian element-wise
		\begin{equation}
			\left(\frac{\partial \statevec{D}^{-}(\statevec{u}_L,	\statevec{u}_R)}{\partial \statevec{u}_R}\Big|_{\statevec{u}_R = \statevec{u}_L}\right)_{ij}
			=
			\frac{\partial \statevec{D}^{-}_{(i)}(\statevec{u}_L, 	\statevec{u}_R)}{\partial {\statevec{u}_R}_{(j)}}\Big|_{\statevec{u}_R = \statevec{u}_L}, \quad i,j = 1,...,n,
		\end{equation}
		where we denote the $i$th component with subscript $(i)$.
		Then, applying the product rule and the identity $\statemat{H}^{-1}(\statevec{u}) = \statevec{w}_{\statevec{u}}$, for each component we obtain
		\begin{equation}
			\begin{aligned}
				\frac{\partial \statevec{D}^{-}_{(i)} (\statevec{u}_L, 	\statevec{u}_R)}{\partial {\statevec{u}_R}_{(j)}}\Big|_{\statevec{u}_R = \statevec{u}_L}
				&=
				\frac{\partial}{\partial {\statevec{u}_R}_{(j)}} 	\left( \sum\limits_{k=1}^n \int_0^1 (1-s) \left(\statemat{A}\,\statemat{H}\right)_{ik}(\Phi(s))\text{d}s \left(\statevec{w}_{(k)}(\statevec{u}_R) - \statevec{w}_{(k)}(\statevec{u}_L)\right)\right)\Bigg|_{\statevec{u}_R = \statevec{u}_L}
				\\&=
				\frac{\partial}{\partial {\statevec{u}_R}_{(j)}} 	\left( \sum\limits_{k=1}^n \int_0^1 (1-s) 	\left(\statemat{A}\,\statemat{H}\right)_{ik}(\Phi(s))\text{d}s\right)\Bigg|_{\statevec{u}_R = \statevec{u}_L} \left(\statevec{w}_{(k)}(\statevec{u}_L) - \statevec{w}_{(k)}(\statevec{u}_L)\right)
				\\ &\quad+
				\int_0^1 (1-s)\text{d}s \sum\limits_{k=1}^n 	\left(\statemat{A}\,\statemat{H}\right)_{ik}(\statevec{w}(\statevec{u}_L)) \frac{\partial \statevec{w}_{(k)}(\statevec{u}_R)}{\partial {\statevec{u}_R}_{(j)}}\Bigg|_{\statevec{u}_R = \statevec{u}_L}
				\\&=
				\frac{1}{2}\sum\limits_{l=1}^n\statemat{A}_{il}(\statevec{u}_L)\sum\limits_{k=1}^n\statemat{H}_{lk}(\statevec{w}(\statevec{u}_L)) \statemat{H}^{-1}_{kj}(\statevec{u}_L)
				\\&=
				\frac{1}{2} \sum\limits_{l=1}^n 	\statemat{A}_{il}(\statevec{u}_L)\delta_{lj}
				=
				\frac{1}{2}\statemat{A}_{ij}(\statevec{u}_L).
			\end{aligned}
		\end{equation}
		
		Lastly, \eqref{eq:ec_condition_3} follows by substitution of $\hat{s} \coloneqq 1 - s$
		\begin{equation}
			\begin{aligned}
				\statevec{D}^-\left(\statevec{u}_L, \statevec{u}_R\right)
				&=
				\int_0^1 \hat{s} \statemat{A}\,\statemat{H}(\statevec{w}_R + \hat{s} \left(\statevec{w}_L - \statevec{w}_R\right)) \jump{\statevec{w}} \text{d}s
				\\&=
				-\int_0^1 \hat{s}\statemat{A}\,\statemat{H}(\statevec{w}_L + \hat{s}\left(\statevec{w}_R - \statevec{w}_L\right)) \jump{\statevec{w}} \text{d}\hat{s}
				=
				-\statevec{D}^+\left(\statevec{u}_R, \statevec{u}_L\right).
			\end{aligned}
		\end{equation}\qed
	\end{proof}

	Similar to Tadmor's entropy fluxes for conservative systems, the integrals in \eqref{eq:definition_ec_flux_strategy_1} may not admit a closed-form solution, even when using a linear path.
	In such cases integrals are approximated numerically with Gauss-quadrature. However, depending on the type of nonlinearity, these quadrature rules can fail to approximate the integral exactly.
	Therefore, the accuracy of the quadrature rule should be chosen high enough to ensure approximation errors remain around machine precision in order to ensure entropy conservation holds.
	\changeTwo{A suitable quadrature rule can therefore be determined by analyzing the approximation error with respect to the nonlinearity of the integrand in \eqref{eq:definition_ec_flux_strategy_1}. 
	Alternatively, the required quadrature rule can be selected heuristically, by analyzing the entropy production for an increasing order of the quadrature rule until entropy conservation is observed.}
	On the other hand, this \changeTwo{accuracy} requirement makes the evaluation of such fluctuations computationally expensive, which motivates the \changeTwo{second} strategy.
	
	\subsubsection{Entropy conservative fluctuations: Strategy 2}\label{sec:ec_strategy_2}	
	An alternative approach common in the design of entropy conservative fluxes for conservation laws is to assume two-point flux functions that are symmetric and consistent and then determine a closed form of the entropy conservative flux from the entropy conservation condition \eqref{eq:ec_condition_4}.
	To generalize this approach to nonconservative systems, we separate conservative and nonconservative contributions in \eqref{eq:ec_condition_2}
	\begin{equation}\label{eq:strat_2_splitting_1}
		\begin{aligned}
			\statevec{D}^-(\statevec{u}_L, \statevec{u}_R) + \statevec{D}^+(\statevec{u}_L, \statevec{u}_R) 
			&= 
			\int_0^1 \statemat{B}(\Phi(s; \statevec{u}_L, \statevec{u}_R) \frac{\partial \Phi(s; \statevec{u}_L, \statevec{u}_R)}{\partial s} \text{d}s
			+ \jump{\statevec{f}},
		\end{aligned}
	\end{equation}
	In contrast to the previous strategy, we now seek a closed form solution to the integral in \eqref{eq:strat_2_splitting_1}.
    Especially in the presence of nonlinear terms, this is simplified if the nonconservative product $\statemat{B}(\statevec{u})\statevec{u}_x$ is expressed in a new set of variables $\statevec{v}(\statevec{u})$ to obtain
	\begin{equation}
		\begin{aligned}
			\statevec{D}^-(\statevec{u}_L, \statevec{u}_R) + \statevec{D}^+(\statevec{u}_L, \statevec{u}_R) 
			&= 
			\int_0^1 \statemat{B}(\Phi(s)) \statevec{u}_{\statevec{v}}(\Phi(s)) \frac{\partial \Phi(s)}{\partial s} \text{d}s
			+ \jump{\statevec{f}}.
		\end{aligned}
	\end{equation}
	where $\Phi(s) \coloneqq \Phi(s; \statevec{u}_L, \statevec{u}_R)$ denotes a linear path with respect to $\statevec{v}(\statevec{u})$
	\begin{equation}
		\Phi(s; \statevec{u}_L, \statevec{u}_R) := \statevec{v}(\statevec{u}_L) + s (\statevec{v}(\statevec{u}_R) - \statevec{v}(\statevec{u}_L)).
	\end{equation}
	
	Analogous to the form of the EC fluctuations \eqref{eq:definition_ec_flux_strategy_1}, we assume that fluctuations have the structure
	\begin{equation}
		\begin{aligned}
		\statevec{D}^{\pm}(\statevec{u}_L, \statevec{u}_R) &= \frac{1}{2}\bar{\statemat{B}}^{\pm}(\statevec{u}_L, \statevec{u}_R)\jump{\statevec{v}} \pm \left(\statevec{f}^{\pm} - \statevec{f}^{*}(\statevec{u}_L, \statevec{u}_R)\right),
		\end{aligned}
		\label{eq:definition_ec_flux_strategy_2}
	\end{equation}
	where $\bar{\statemat{B}}^{\pm} : \stateset \times \stateset \rightarrow \mathbb{R}^{n \times n}$ is a two-point matrix-valued function satisfying
	\begin{align}
		2\bar{\statemat{B}}^{\pm}(\statevec{u}, \statevec{u}) &= \statemat{B}(\statevec{u})\statevec{u}_{\statevec{v}}(\statevec{u})\label{eq:fluctuation_strat2_cond1}, 
		\\
		\bar{\statemat{B}}^{-}(\statevec{u}_R, \statevec{u}_L) 
		&=
		\bar{\statemat{B}}^{+}(\statevec{u}_L, \statevec{u}_R),
		\label{eq:fluctuation_strat2_cond2}
		\\
		\bar{\statemat{B}}^{-}(\statevec{u}_L, \statevec{u}_R) 
		+
		\bar{\statemat{B}}^{+}(\statevec{u}_L, \statevec{u}_R)
		&=
		\int_0^1 \statemat{B}(\Phi(s))\statevec{u}_{\statevec{v}}(\Phi(s)) \text{d}s.
		\label{eq:fluctuation_strat2_cond3}
	\end{align}
	Furthermore, $\statevec{f}^* : \stateset \times \stateset \rightarrow \mathbb{R}^n$ is a symmetric and consistent two-point flux function
	\begin{equation}\label{eq:consistency_condition_cons_flux}
		\begin{aligned}
			\statevec{f}^*(\statevec{u}, \statevec{u}) &= \statevec{f}(\statevec{u}),\\
			\statevec{f}^*(\statevec{u}_L, \statevec{u}_R) &= \statevec{f}^*(\statevec{u}_R, \statevec{u}_L),
		\end{aligned}
	\end{equation} 
	and we introduce the notation $\statevec{f}^- = \statevec{f}(\statevec{u}_L)$ and $\statevec{f}^{+} = \statevec{f}(\statevec{u}_R)$ to denote the conservative flux evaluated at the left and right solution states.
	
	With this structure we can then single out those fluctuations that provide an entropy conservative discretization by substituting the ansatz \eqref{eq:definition_ec_flux_strategy_2} into the entropy conservation condition \eqref{eq:ec_condition_4}, which yields
	\begin{equation}
		\jump{q - \statevec{w}^T\statevec{f}} = \statevec{w}_R^T\bar{\statemat{B}}^+(\statevec{u}_L, \statevec{u}_R)\jump{\statevec{v}} + \statevec{w}_L^T\bar{\statemat{B}}^-(\statevec{u}_L, \statevec{u}_R)\jump{\statevec{v}} - \jump{\statevec{w}}^T\statevec{f}^*(\statevec{u}_L, \statevec{u}_R).
		\label{eq:entropy_conservation_condition_strategy_2}
	\end{equation}

	% Consistency proof
	\begin{theorem}\label{thm:ec_fluxes_strat_2}
		The fluctuations 
		\begin{equation*}
				\statevec{D}^{\pm}(\statevec{u}_L, \statevec{u}_R) = \bar{\statemat{B}}^{\pm}(\statevec{u}_L, \statevec{u}_R)\jump{\statevec{v}} \pm \left(\statevec{f}^{\pm} - \statevec{f}^{*}(\statevec{u}_L, \statevec{u}_R)\right),
		\end{equation*}
		satisfying \eqref{eq:fluctuation_strat2_cond1}, \eqref{eq:fluctuation_strat2_cond2}, \eqref{eq:fluctuation_strat2_cond3}, \eqref{eq:consistency_condition_cons_flux} and \eqref{eq:entropy_conservation_condition_strategy_2}
		are entropy conservative fluctuations according to Definition \ref{def:ec_fluctuation}.
	\end{theorem}
	\begin{proof}
		The properties \eqref{eq:ec_condition_1}, \eqref{eq:ec_condition_2}, \eqref{eq:ec_condition_3}, \eqref{eq:ec_condition_4} follow by construction, so we only need to show \eqref{eq:ec_condition_5}. For this we consider the nonconservative and conservative contributions separately
		\begin{equation}
			\left(\frac{\partial \statevec{D}^{-}(\statevec{u}_L, \statevec{u}_R)}{\partial \statevec{u}_R}\Big|_{\statevec{u}_R = \statevec{u}_L}\right)
			=
			\frac{\partial \bar{\statemat{B}}^{-}(\statevec{u}_L, \statevec{u}_R) \jump{\statevec{v}}}{\partial \statevec{u}_R}\Big|_{\statevec{u}_R = \statevec{u}_L}
			+
			\frac{\partial \statevec{f}^*\left(\statevec{u}_L, \statevec{u}_R\right)}{\partial \statevec{u}_R}\Big|_{\statevec{u}_R = \statevec{u}_L}.
		\end{equation}
		For the conservative contribution it was shown in \cite{ranocha2018generalised} that 
		\begin{equation}
			\frac{\partial \statevec{f}^*\left(\statevec{u}_L, \statevec{u}_R\right)}{\partial \statevec{u}_R}\Big|_{\statevec{u}_R = \statevec{u}_L} 
			=
			\frac{1}{2} \frac{\partial \statevec{f}(\statevec{u})}{\partial \statevec{u}}\Big|_{\statevec{u} = \statevec{u}_L}
			\label{eq:accuracy_conservative_part}
		\end{equation}
		
		For the nonconservative part we evaluate the Jacobian element-wise, by applying the product rule and using the consistency \eqref{eq:fluctuation_strat2_cond1} to obtain
		\begin{equation}
			\begin{aligned}
			\frac{\partial \left(\bar{\statemat{B}}^{-}(\statevec{u}_L, \statevec{u}_R) \jump{\statevec{v}}\right)_{(i)}}{\partial {\statevec{u}_R}_{(j)}}\Big|_{\statevec{u}_R = \statevec{u}_L}
			&=
			\frac{\partial}{\partial {\statevec{u}_R}_{(j)}} \left(\sum\limits_{k=1}^n \bar{\statemat{B}}^-_{ik}(\statevec{u}_L, \statevec{u}_R)\left( {\statevec{v}_R}_{(k)} - {\statevec{v}_L}_{(k)}\right)\right)\Big|_{\statevec{u}_R = \statevec{u}_L}
			\\&=
			\sum\limits_{k=1}^n \frac{\partial \bar{\statemat{B}}^-_{ik}(\statevec{u}_L, \statevec{u}_R)}{\partial {\statevec{u}_R}_{(j)}}\Big|_{\statevec{u}_R = \statevec{u}_L}  \left( {\statevec{v}_L}_{(k)} - {\statevec{v}_L}_{(k)}\right)
			+
			\bar{\statemat{B}}_{ij}^-(\statevec{u}_L, \statevec{u}_L)
			\frac{\partial \statevec{v}_{(k)}(\statevec{u}_R)}{\partial {\statevec{u}_R}_{(j)}}\Bigg|_{\statevec{u}_R = \statevec{u}_L}
			\\&=
			\frac{1}{2}\sum\limits_{l=1}^n\statemat{B}_{il}(\statevec{u}_L)\sum\limits_{k=1}^n(\statevec{u}_{\statevec{v}}(\statevec{v}(\statevec{u}_L)))_{lk} (\statevec{v}_{\statevec{u}}(\statevec{u}_L))_{kj}
			\\&=
			\frac{1}{2} \sum\limits_{l=1}^n 	\statemat{B}_{il}(\statevec{u}_L)\delta_{lj}
			=
			\frac{1}{2} \statemat{B}_{ij}(\statevec{u}_L).
			\end{aligned}
			\label{eq:accuracy_noncons_part}
		\end{equation}
		Finally combining the contributions \eqref{eq:accuracy_conservative_part} and \eqref{eq:accuracy_noncons_part} yields the desired result
		\begin{equation}
			\left(\frac{\partial \statevec{D}^{-}(\statevec{u}_L, \statevec{u}_R)}{\partial \statevec{u}_R}\Big|_{\statevec{u}_R = \statevec{u}_L}\right)
			=
			\frac{1}{2}\left(\statemat{B}(\statevec{u}) + \frac{\partial \statevec{f}(\statevec{u})}{\partial \statevec{u}}\right)\Big|_{\statevec{u} = \statevec{u}_L}
			=
			\frac{1}{2}\statemat{A}(\statevec{u}_L) .
		\end{equation}
			\qed
	\end{proof}

\subsection{Entropy stable DGSEM}
	While the entropy conservation law applies for smooth solutions, at discontinuities entropy should be dissipated.
	As such we want to design a method that is entropy stable in the sense that it recovers an integral version of the entropy inequality \eqref{eq:entropy_inequality} in the semi-discrete case.
	
	To achieve this, we use the entropy conservative DGSEM developed in the previous section as a baseline scheme to which we add dissipation at element interfaces to recover an entropy inequality.
	This is achieved by replacing the entropy conservative interface fluctuations in  \eqref{eq:flux_diff_DGSEM}, with fluctuations that are entropy stable.
	\begin{definition}
		A fluctuation \changeOne{$\statevec{D}_{ES}^{\pm}$ is defined as entropy stable} if it satisfies
		\begin{equation}
			\statevec{D}^{\pm}_{ES}(\statevec{u}_L, \statevec{u}_R) = \statevec{D}^{\pm}_{EC}(\statevec{u}_L, \statevec{u}_R) \pm \statemat{Q}(\statevec{u}_L, \statevec{u}_R)\jump{\statevec{u}},
		\end{equation}
		where $\statevec{D}^{\pm}_{EC}$ is an entropy conservative fluctuations according to Definition \ref{def:ec_fluctuation} and $\statemat{Q}: \stateset \times \stateset \rightarrow \mathbb{R}^{n \times n}$ is a numerical viscosity matrix satisfying the inequality
		\begin{equation}
			\jump{\statevec{w}}^T \statemat{Q}(\statevec{u}_L, \statevec{u}_R) \jump{\statevec{u}} \geq 0.
			\label{eq:es_dissipation_inequality}
		\end{equation}
	\end{definition}
	
	\begin{theorem}\label{thm:es_dgsem}
		The nonconservative DGSEM in flux differencing formulation 
		\begin{equation*}
			\omega_i \frac{\Delta x_k}{2} \frac{\partial \statevec{U}_i^k}{\partial t} 
			+
			\omega_i \sum\limits_{m=0}^N 2\mathcal{D}_{im} \statevec{D}^-_{EC}(\statevec{U}_i^k,\statevec{U}_m^k)
			+ \delta_{i0}\statevec{D}_{ES}^+(\statevec{U}_N^{k-1}, \statevec{U}_0^k)
			+ \delta_{iN}\statevec{D}^-_{ES}(\statevec{U}_N^{k}. \statevec{U}_0^{k+1}) = 0,
		\end{equation*}
		with entropy conservative volume fluctuations $\statevec{D}^{\pm}_{EC}$ and entropy stable interface fluctuations $\statevec{D}^{\pm}_{ES}$ recovers the semi-discrete entropy inequality at the element level
		\begin{equation}
			\sum\limits_{i=0}^N \omega_i\frac{\Delta x_k}{2} \frac{\partial S(\statevec{U}_i^k)}{\partial t} 
			\leq
			Q(\statevec{U}_N^{k-1}, \statevec{U}_0^k)
			-Q(\statevec{U}_N^k, \statevec{U}_0^{k+1}),
		\end{equation}
		where $Q$ denotes a numerical entropy flux, consistent with the physical entropy flux $Q(\statevec{u}, \statevec{u}) = q(\statevec{u})$.
	\end{theorem}
	\begin{proof}
		We begin with the result from Lemma \ref{lemma:ec_dgsem} and use the relations $(\cdot)_L = \avg{\cdot} - \frac{1}{2}\jump{\cdot}$,  $(\cdot)_R = \avg{\cdot} + \frac{1}{2}\jump{\cdot}$ and \eqref{eq:es_dissipation_inequality} to obtain
		\begin{equation}
			\begin{aligned}
				\sum\limits_{i=0}^N \omega_i\frac{\Delta x_k}{2} \frac{\partial S(\statevec{U}_i^k)}{\partial t} 
				&= 
				\left(q_0^k - {\statevec{W}_0^k}^T\statevec{D}^+_{ES}(\statevec{U}_N^{k-1}, \statevec{U}_0^k)\right)
				-\left(q_N^k + {\statevec{W}_N^k}^T\statevec{D}^-_{ES}(\statevec{U}_N^{k}, \statevec{U}_0^{k+1})\right)
				\\
				&=
				\left(q_0^k 
				- {\statevec{W}_0^k}^T\statevec{D}^+_{EC}(\statevec{U}_N^{k-1}, \statevec{U}_0^k)
				-
				{\statevec{W}_0^k}^T\statemat{Q}(\statevec{U}_N^{k-1}, \statevec{U}_0^k)(\statevec{U}_0^{k} - \statevec{U}_N^{k-1})\right)
				\\
				&-\left(q_N^k 
				+ {\statevec{W}_N^k}^T\statevec{D}^-_{EC}(\statevec{U}_N^{k}, \statevec{U}_0^{k+1})
				-
				{\statevec{W}_N^k}^T\statemat{Q}(\statevec{U}_N^{k}, \statevec{U}_0^{k+1})(\statevec{U}_0^{k+1} - \statevec{U}_N^{k})\right)
				\\
				&=
				\left(q_0^k 
				- {\statevec{W}_0^k}^T\statevec{D}^+_{EC}(\statevec{U}_N^{k-1}, \statevec{U}_0^k)
				-
				\left({\statevec{W}_N^{k-1}} + \statevec{W}_0^{k}\right)^T\frac{1}{2}\statemat{Q}(\statevec{U}_N^{k-1}, \statevec{U}_0^k)(\statevec{U}_0^{k} - \statevec{U}_N^{k-1})\right)
				\\
				&-\left(q_N^k 
				+ {\statevec{W}_N^k}^T\statevec{D}^-_{EC}(\statevec{U}_N^{k}, \statevec{U}_0^{k+1})
				-
				\left({\statevec{W}_N^k} + \statevec{W}_0^{k+1}\right)^T\frac{1}{2}\statemat{Q}(\statevec{U}_N^{k}, \statevec{U}_0^{k+1})(\statevec{U}_0^{k+1} - \statevec{U}_N^{k})\right)
				\\
				&-
				\left({\statevec{W}_0^k} - \statevec{W}_N^{k-1}\right)^T\frac{1}{2}\statemat{Q}(\statevec{U}_N^{k-1}, \statevec{U}_0^k)(\statevec{U}_0^{k} - \statevec{U}_N^{k-1})
				-
				\left({\statevec{W}_0^{k+1}} - \statevec{W}_N^{k}\right)^T\frac{1}{2}\statemat{Q}(\statevec{U}_N^{k}, \statevec{U}_0^{k+1})(\statevec{U}_0^{k+1} - \statevec{U}_N^{k})
				\\&\leq
				Q(\statevec{U}_N^{k-1}, \statevec{U}_0^k)
				-Q(\statevec{U}_N^k, \statevec{U}_0^{k+1}),
			\end{aligned}
		\end{equation}
		with numerical entropy fluxes
		\begin{equation}
			\begin{aligned}
			Q(\statevec{U}_N^{k-1}, \statevec{U}_0^k)
			&=
			\left(q_0^k 
			+ {\statevec{W}_0^k}^T\statevec{D}^+_{EC}(\statevec{U}_N^{k-1}, \statevec{U}_0^k)
			-
			\left({\statevec{W}_N^{k-1}} + \statevec{W}_0^{k}\right)^T\frac{1}{2}\statemat{Q}(\statevec{U}_N^{k-1}, \statevec{U}_0^k)(\statevec{U}_0^{k} - \statevec{U}_N^{k-1})\right)
			\\&=
			\left(q_N^{k-1} 
			- {\statevec{W}_N^{k-1}}^T\statevec{D}^-_{EC}(\statevec{U}_N^{k-1}, \statevec{U}_0^{k})
			-
			\left({\statevec{W}_N^{k-1}} + \statevec{W}_0^{k}\right)^T\frac{1}{2}\statemat{Q}(\statevec{U}_N^{k-1}, \statevec{U}_0^k)(\statevec{U}_0^{k} - \statevec{U}_N^{k-1})\right).
			\end{aligned}
		\end{equation}
		\qed
	\end{proof}
	
    \begin{lemma}
    	Let $S(\statevec{u})$ be a convex entropy function, with respective entropy variables $\statevec{w} = S_{\statevec{u}}$, then for any $\statevec{u}_L, \statevec{u}_R \in \stateset$ the local Lax-Friedrichs (LLF) type dissipation given by
    	\begin{equation}
    		\statemat{Q}_{llf}(\statevec{u}_L, \statevec{u}_R) = \frac{1}{2}|\lambda|_{\max}\statemat{I}
    	\end{equation}
    	dissipates entropy as it satisfies the inequality \eqref{eq:es_dissipation_inequality}.
    \end{lemma}
	\begin{proof}
		From convexity of $S(\statevec{u})$ we obtain the inequalities
		\begin{equation}
			\begin{aligned}
			S(\statevec{u}_R) \geq S(\statevec{u}_L) + \nabla S(\statevec{u}_L)^T(\statevec{u}_R - \statevec{u}_L),\\
			S(\statevec{u}_L) \geq S(\statevec{u}_R) + \nabla S(\statevec{u}_R)^T(\statevec{u}_L - \statevec{u}_R).
			\end{aligned}
		\end{equation}
		Adding both inequalities, and using that \changeOne{$\nabla S(\statevec{u}) = \statevec{w}$}, yields
		\begin{equation}
			0 \leq (\nabla S(\statevec{u}_R)^T - \nabla S(\statevec{u}_L)^T)(\statevec{u}_R - \statevec{u}_L) = \jump{\statevec{w}}^T\jump{\statevec{u}} 
		\end{equation}
		\qed
	\end{proof}
    
    For conservative systems, more sophisticated entropy dissipation terms such as entropy stable Roe or HLL type dissipation have been constructed \cite{schmidtmann2017hybrid}.
    However, these formulations rely on the symmetry of $\statemat{A}\,\statemat{H}$ and the related identity, demonstrated in \cite{barth1999numerical}, that for conservative systems the entropy Hessian is a product of scaled eigenvectors of the flux Jacobian.
    As discussed in Section~\ref{sec:continuous_entropy_analysis}, for nonconservative systems the entropy Hessian $\statemat{H}$ no longer acts as a symmetrizer, hence the classical matrix-valued dissipation terms may fail to preserve the entropy inequality.
    On the other hand LLF-type dissipation is known to be overly dissipative and in some instances fails to preserve important steady state solutions.
    
    To overcome these issues, we propose a simple blending procedure to construct entropy stable dissipation.
    The idea is that it is possible to make any matrix-valued dissipation term entropy stable by introducing a convex blending with LLF dissipation.
    Since the LLF dissipation is known to be entropy dissipative, we then blend in just enough of the safe dissipation to ensure \eqref{eq:es_dissipation_inequality} holds.
  
    Assume that we have some numerical viscosity matrix $\statemat{Q}$, that is not necessarily entropy dissipative. 
    Then we introduce the convex blending 
    \begin{equation}
    	\statemat{Q}_{ES} \coloneqq \alpha \statemat{Q}_{llf} + \left(1 - \alpha\right) \statemat{Q}, \quad \alpha \in [0, 1]
    	\label{eq:ansatz_blended_dissipation}
    \end{equation}
    to construct the blended entropy stable numerical viscosity matrix $\statemat{Q}_{ES}$.
    Depending on the blending factor $\alpha$, this procedure recovers either $\statemat{Q}$ for $\alpha = 0$, the entropy dissipative LLF viscosity matrix $\statemat{Q}_{llf}$ for $\alpha = 1$, or any convex combination thereof for $0<\alpha<1$.
    To determine the choice of $\alpha$, we first determine the entropy production of the viscosity matrices
	\begin{equation}
		\Delta S = \jump{\statevec{w}}^T \statemat{Q}\jump{\statevec{u}}, \quad \Delta S_{\!llf} = \jump{\statevec{w}}^T \statemat{Q}_{llf}\jump{\statevec{u}}.
	\end{equation}
	Substituting the ansatz \eqref{eq:ansatz_blended_dissipation} into the inequality \eqref{eq:es_dissipation_inequality} then provides us with the following condition
	\begin{equation}
		\alpha \Delta S_{\!llf} + (1-\alpha) \Delta S \,\changeOne{\geq}\, 0,
	\end{equation}
	which finally determines our choice of the blending parameter to be
	\begin{equation}\label{eq:definition_alpha}
		\changeOne{\alpha = 
		\begin{cases}
			0  & \text{for } \Delta S \geq 0, \\
			\changeOne{ \frac{\Delta S}{\Delta S - \Delta S_{\!llf}} } & \text{else.}
		\end{cases}}
	\end{equation}

	Note, that in the case of entropy production \changeOne{($\Delta S < 0$)}, the entropy dissipative property of the LLF dissipation $\Delta S_{\!llf} \geq 0$ ensures a \changeOne{negative}, non-zero denominator. \changeOne{Furthermore, in this case the relation $\Delta S - \Delta S_{llf} < \Delta S < 0$ yields $0 < \alpha \leq 1$, which ensures that the blending is convex.}
	A pseudocode description of this blending procedure is presented in Algorithm~ \ref{alg:blending_procedure}.  
	
	\begin{algorithm}[h]
		\caption{Entropy stable dissipation blending}
		\label{alg:blending_procedure}
		\KwIn{$\statevec{u}_R, \; \statevec{u}_L, \; \statemat{Q}, \; \statemat{Q}_{llf}$}
		\KwOut{$\statemat{Q}_{ES}$}
		
		$\statevec{w}_{L,R} \gets \statevec{w}(\statevec{u}_{L,R})$\;
		$\jump{\statevec{w}} \gets \statevec{w}_R - \statevec{w}_L$\;
		$\Delta S \gets \jump{\statevec{w}}^T \statemat{Q}\jump{\statevec{u}}$\;
		\uIf{\changeOne{$\Delta S \geq 0 \;$}}{
			$\statemat{Q}_{ES} \gets \statemat{Q}$
		}
		\Else{
			$\Delta S_{\!llf} \gets \jump{\statevec{w}}^T \statemat{Q}_{llf}\jump{\statevec{u}}$\;
			$\alpha \gets \changeOne{\Delta S \; / \; (\Delta S - \Delta S_{\!llf})}$\;
			$\statemat{Q}_{ES} \gets \alpha \statemat{Q}_{llf} + (1-\alpha) \statemat{Q}$
		}
	\end{algorithm}

\section{Saint-Venant-Exner system}\label{sec:SVE_system}
We apply the numerical method developed in Section~\ref{sec:numerical_method} to construct an entropy stable, high-order approximation to the Saint-Venant-Exner (SVE) system.
This system couples a fluid layer modeled by the shallow water equations with the Exner equation for sediment transport.
A shortcoming of the classical SVE system is the absence of an auxiliary entropy function, which prevents the construction of ES discretizations.
This issue was remedied in \cite{fernandez2017formal}, where the model was rederived to introduce a correction term that provides a system with an associated entropy function.
As we aim to design an ES method, we consider the model with correction factor from \cite{fernandez2017formal}
\begin{equation}\label{eq:sve_system}
	\begin{cases}
		\partial_t h + \partial_x hv = 0, \\
		\partial_t hv + \partial_x hv^2 + gh\partial_x (h + b) + \frac{1}{r}gh_b\partial_x (rh + b) + \frac{\tau}{\rho_f} = 0,\\
		\partial_t b + \partial_x q_b = 0,
	\end{cases}
\end{equation}
where $h(x,t)$ denotes the water height, $hv(x,t)$ the momentum, $b(x,t)$ the sediment height, $g$ is the gravitational acceleration, and $r = \rho_f / \rho_s$ is the ratio between fluid $\rho_f$, and sediment density $\rho_s$.
Furthermore, $\tau$ is the shear stress at the water-sediment interface, defined by an empirical friction law. In this work, we consider the Manning friction term
\begin{equation}
	\tau = \frac{\rho_f g n^2 v|v|}{h^{1/3}},
\end{equation}
with Manning coefficient $n$.
The sediment discharge, denoted $q_b(h, hv)$, is specified through empirical sediment discharge formulas. For convenience, we define the active sediment height $h_b \coloneqq \frac{q_b}{v}$.
While a vast number of discharge formulas have been proposed in the literature, we focus the discussion on two widely used models from Grass \cite{grass1981sediment} and Meyer-Peter and Müller (MPM) \cite{meyer1948formulas}.
The Grass model assumes a power law relationship with the velocity
\begin{equation}\label{eq:grass_formula}
	\frac{q_b}{\vartheta} = A_g v^3,
\end{equation}
where $A_g \in [0,1]$ is an empirical coefficient characterizing the sediment properties that is estimated from experimental data and $\vartheta = 1 / (1 - \varphi)$ is defined by the sediment porosity $\varphi$. Due to the direct relationship with the velocity a core assumption of the Grass model is that sediment transport begins immediately for non-zero velocities.
This unphysical behavior has been improved in other formulas, like MPM, that restrict sediment transport when the shear stress falls below a certain threshold
\begin{equation}\label{eq:mpm_formula}
	\frac{q_b}{\vartheta} = \text{sgn}(\tau) 8(\theta - \theta_c)^\frac{3}{2}_+ \sqrt{g \left(\frac{1}{r} - 1\right) d_s^3},
\end{equation}
with the dimensionless shields parameter $\theta = \frac{\tau}{(\rho_s - \rho_f)gd_s}$, the critical shields parameter $\theta_c$ (typically set to $\theta_c = 0.047$) and the sediment grain diameter $d_s$.

% Well-balancedness
Analogous to the shallow water system, the SVE system admits steady-state solutions that satisfy $\statemat{A}\statevec{u}_x = 0$, with \changeThree{the vector of solution variables $\mathbf{u}(x,t) := \left( h, hv, b\right)^T$ and} the generalized Jacobian
\begin{equation}
	\statemat{A}(\statevec{u}) = \begin{pmatrix}
		0 & 1 & 0 \\ g(h + h_b)- v^2 & 2v & g(h+\frac{1}{r}h_b) \\ \frac{\partial q_b}{\partial h} & \frac{\partial q_b}{\partial hv} & 0
	\end{pmatrix}.
\end{equation}
One particularly important example is the lake-at-rest steady state
\begin{equation}
	\mathcal{U}_{wb}\coloneqq \{\statevec{u}(x, t) : h(x,t) + b(x,t) = H_0, \, hv(x,t) = 0, \quad t \geq 0\},
	\label{eq:lake-at-rest conditions}
\end{equation}
for some constant water level $H_0 \in \mathbb{R}$.
Preserving the lake-at-rest steady state, commonly referred to as the well-balanced property, is crucial for numerical methods as violations of the well-balanced property lead to unphysical oscillations on the order of the mesh spacing \cite{chertock2018well}.

The system \eqref{eq:sve_system} also admits an additional conservation law in the form of the entropy–entropy flux pair
\begin{equation}
	\begin{aligned}
		S(\statevec{u}) = \frac{1}{2}rhv^2 + \frac{1}{2}g(rh^2 + b^2) + rghb,
		\quad
		q(\statevec{u}) = rhv\left(\frac{v^2}{2} + g(h + b)\right) + gq_b(rh+b).
	\end{aligned}
\end{equation}
Differentiating the entropy function with respect to the solution variables $\statevec{u}$ yields the entropy variables
\begin{equation}
	\statevec{w}(\statevec{u}) = \begin{pmatrix}
		r(g(h + b) - \frac{v^2}{2}) \\
		rv\\
		g(rh+b)
	\end{pmatrix},
\end{equation}
which are key for the entropy analysis. The existence of such an associated entropy pair is important to ensure physically consistent solutions and nonlinear stability.
In the following section we develop a numerical method that preserves both continuous invariants in the semi-discrete case.

\subsection{Numerical method}
We apply the nonconservative DGSEM with flux differencing formulation introduced in Section~\ref{sec:numerical_method} to construct a high-order, entropy stable, and well-balanced numerical approximation of the SVE system.
From Theorem~\ref{thm:es_dgsem}, these properties \changeTwo{follow} immediately with EC fluctuations for the volume terms and ES fluctuations at element interfaces. 
We apply the strategy described in Theorem~\ref{thm:ec_fluxes_strat_2} to derive such EC fluctuations as a baseline and then formulate suitable dissipation terms to obtain the ES fluctuations.

% Entropy conservative fluctuations
To construct the EC fluctuations, the system is split into conservative and nonconservative components
\begin{equation}
	\statevec{f}(\statevec{u}) = \begin{pmatrix}
		hv \\ hv^2 \\ q_b
		\end{pmatrix}
	,\quad
	\statemat{B}(\statevec{u}) = \begin{pmatrix}
		0 & 0 & 0 \\ g(h + h_b) & 0 & g(h + \frac{1}{r} h_b) \\ 0 & 0 & 0
	\end{pmatrix}.
\end{equation}
Next, to determine the nonconservative part of the fluctuation, the path integral in \eqref{eq:fluctuation_strat2_cond3} must be evaluated in closed form.
This is complicated, however, by the fact that $h_b(\statevec{u})$ may depend on both the friction law and sediment discharge formulation, and is generally a highly nonlinear term.
Nevertheless, we can simplify the evaluation of the path integral by introducing the auxiliary variables
\begin{equation}
	\statevec{v}(\statevec{u}) := 
	\begin{pmatrix} h \\ h_b \\ b \end{pmatrix},
	\quad
	\statemat{B}(\statevec{v}) = 
	\begin{pmatrix}
		0 & 0 & 0 \\ g(h + h_b) & 0 & g(h + \frac{1}{r} h_b) \\ 0 & 0 & 0
	\end{pmatrix}
	,\quad
	\statevec{u}_{\statevec{v}} =
	\begin{pmatrix}
		1 & 0 & 0 \\
		-\frac{\partial h_b}{\partial h} / \frac{\partial h_b}{\partial hv} & 1 / \frac{\partial h_b}{\partial hv} & 0 \\
		0 & 0 & 1
	\end{pmatrix}.
\end{equation}
We find, assuming a linear path with respect to $\statevec{v}$, that
\begin{equation}
	\bar{\statemat{B}}^{-}(\statevec{u}_L, \statevec{u}_R) 
	+
	\bar{\statemat{B}}^{+}(\statevec{u}_L, \statevec{u}_R) 
	=
	\int_0^1 \statemat{B}(\statevec{v})\statevec{u}_{\statevec{v}}\text{d}s
	=
	\begin{pmatrix}
		0 & 0 & 0 \\ g(\avg{h} + \avg{h_b}) & 0 & g(\avg{h} + \frac{1}{r}\avg{h_b}) \\ 0 & 0 & 0
	\end{pmatrix}.
\end{equation}

\changeThree{Next, to determine a fluctuation that is EC according to Theorem~\ref{thm:ec_fluxes_strat_2}, we must determine the components $\statevec{f}^*$ and $\bar{\statemat{B}}^{\pm}$, such that their combination satisfies the entropy conservation condition \eqref{eq:entropy_conservation_condition_strategy_2}.
Including the additional assumption on the structure of $\bar{\statemat{B}}^{\pm}$ given in \eqref{eq:fluctuation_strat2_cond1} and \eqref{eq:fluctuation_strat2_cond2} provides us with two possible candidates $\bar{\statemat{B}}^- = \frac{1}{2}\statemat{B}(\statevec{u}_L)$ and $\bar{\statemat{B}}^+ = \frac{1}{2}\statemat{B}(\statevec{u}_R)$ or $\bar{\statemat{B}}^\pm(\statevec{u}_L, \statevec{u}_R) = \frac{1}{2}\statemat{B}(\avg{\statevec{u}})$ for the nonconservative part. The respective nonconservative term then needs to be combined with a numerical flux $\statevec{f}^*$, that is consistent and symmetric as defined in \eqref{eq:consistency_condition_cons_flux}. Analogous to the case of EC fluxes for conservative systems, this choice is in general non-unique and an admissible combination must be determined heuristically, as no systematic construction procedure is available. For the remainder of this work, we consider the admissible choice}
\begin{equation}
	\statevec{D}_{EC}^{\pm} = \bar{\statemat{B}}^{\pm}\jump{\statevec{v}} \pm (\statevec{f}^{\pm} - \statevec{f}^*),
	\quad
	\statevec{f}^* = 
	\begin{pmatrix} \avg{hv} \\ \avg{hv}\avg{v} \\ \avg{q_b}
	\end{pmatrix},
	\quad \bar{\statemat{B}}^- = \frac{1}{2}\statemat{B}(\statevec{u}_L),
	\quad
	\bar{\statemat{B}}^+ = \frac{1}{2}\statemat{B}(\statevec{u}_R).
	\label{eq:ec_fluctuation_sve}
\end{equation}
Note, that this formulation does not assume a specific sediment transport formula.
Therefore, it provides EC fluctuations for any formula that is written in terms of the sediment discharge.

% ES fluctuations
From the EC fluctuation, we construct an ES fluctuation by adding a suitable dissipation term that satisfies \eqref{eq:es_dissipation_inequality} while preserving the lake-at-rest steady state.
Although LLF dissipation guarantees entropy stability, it is known to violate the lake-at-rest equilibrium and, therefore, does not meet our requirements.
Conversely, Roe-type dissipation has been applied successfully for the SVE system in \cite{diaz2009two, hudson2001numerical, cordier2011bedload, carraro2018efficient}, but it may lead to entropy-violating solutions \cite{barth1999numerical}.
Entropy stable variants of Roe dissipation have been developed for conservative systems \cite{schmidtmann2017hybrid}, but rely on entropy symmetrization.
Hence, this strategy cannot be applied for nonconservative systems.
As a remedy, we introduce a blended dissipation term that combines Roe and LLF dissipation following the blending procedure described in Section~{\ref{sec:numerical_method} to obtain a dissipation term that respects both entropy stability and well-balancedness.

% Construction of the Roe matrix
To build the blended dissipation term, we introduce the Roe dissipation defined by the viscosity matrix
$\statemat{Q}_{roe} := \frac{1}{2}|\statemat{A}_{roe}| = \frac{1}{2}\statemat{R}|\statemat{\Lambda}|\statemat{R}^{-1}$ 
where $\statemat{R}$ is the matrix of right eigenvectors and $|\statemat{\Lambda}| := \text{diag}(|\lambda_1|, ..., |\lambda_m|)$ the absolute value matrix of the eigenvalues corresponding to the eigendecomposition of the Roe matrix $\statemat{A}_{roe} := \statemat{A}(\tilde{\statevec{u}})$. 
Since exact Roe averages depend on the specific sediment and friction models, obtaining a general closed-form solution is infeasible.
Therefore, we use the approximate Roe averages
\begin{equation}
	\begin{aligned}
		\tilde{\statevec{u}}(\statevec{u}_L, \statevec{u}_R) \coloneqq
		\left(
		\tilde{h} , \; \tilde{h}\tilde{v} , \; \tilde{b}
		\right)^T
		\approx
		\left(\avg{h} , \;
		\avg{h}\frac{\sqrt{h_L}v_L + \sqrt{h_R}v_R}{\sqrt{h_L} + \sqrt{h_R}} , \;
		\avg{b}\right)^T.
	\end{aligned}
\end{equation}

We compute $|\statemat{A}_{roe}|$ following the procedure from \cite{carraro2018efficient} where the eigenvalues of $\statemat{A}_{roe}$ are determined with Cardano's formula.
To compute the derivatives of the sediment discharge, we use automatic differentiation.
Then eigenvectors are recovered as solutions to the eigenvalue problem $\statemat{A}_{roe}\statevec{r}_i = \lambda_i \statevec{r}_i$.
This yields right eigenvectors
\begin{equation}
	\statemat{R} = \begin{pmatrix}
		\statevec{r}_1, \statevec{r}_2, \statevec{r_3}
	\end{pmatrix},
	\quad
	\statevec{r}_i = \begin{pmatrix}
		1, & \lambda_i, & \frac{(\tilde{v} - \lambda_i)^2 - c_1}{c_2}
	\end{pmatrix}^T.
\end{equation}
Analogously, one obtains the inverse matrix
\begin{equation}
	\statemat{R}^{-1} = \begin{pmatrix}
		\statevec{l}_1, \statevec{l}_2, \statevec{l_3}
	\end{pmatrix}^T,
	\quad
	\statevec{l}_i = \begin{pmatrix}
		\dfrac{c_1 - \tilde{v}^2 + \lambda_j \lambda_k}{(\lambda_i - \lambda_j)(\lambda_i - \lambda_k)},&
		\dfrac{2 \tilde{v} - \lambda_j - \lambda_k}{(\lambda_i - \lambda_j)(\lambda_i - \lambda_k)},&
		\dfrac{c_2}{(\lambda_i - \lambda_j)(\lambda_i - \lambda_k)}
	\end{pmatrix},
\end{equation}
with coefficients $c_1 = g (\tilde{h} + \changeOther{\tilde{h}_b})$, $c_2 = g (\tilde{h} + \frac{\changeOther{\tilde{h}_b}}{r})$ \changeOther{where $\tilde{h}_b = \frac{q_b(\tilde{\statevec{u}})}{\tilde{v}}$} and cyclic indices $i,j,k$.

Finally, the ES viscosity matrix is obtained through the blending
\begin{equation}
	\statemat{Q}_{ES-roe} = \alpha \statemat{Q}_{llf} + (1-\alpha) \statemat{Q}_{roe},
\end{equation}
where $\alpha \in [0,1]$ is a blending parameter determined by \eqref{eq:definition_alpha}, which controls the contribution of each dissipation operator.
The corresponding ES fluctuations are
\begin{equation}\label{eq:es_roe_fluctuations}
	\statevec{D}^{\pm}_{ES} = \statevec{D}^{\pm}_{EC} \pm \statemat{Q}_{ES-roe}\jump{\statevec{u}}.
\end{equation}

% Well-balancedness
\subsection{Well-balancedness}
To verify that the scheme preserves the lake-at-rest equilibrium, we evaluate both the EC fluctuations \eqref{eq:ec_fluctuation_sve} and ES fluctuations \eqref{eq:es_roe_fluctuations} under lake-at-rest conditions \eqref{eq:lake-at-rest conditions}.
We first observe that \eqref{eq:ec_fluctuation_sve} vanishes directly when evaluated for the lake-at-rest steady state
\begin{equation}
	 \statevec{D}^{\pm}_{EC}(\statevec{u}_L, \statevec{u}_R) = h_{(L,\,R)}\jump{h+b} = 0, \quad \forall \statevec{u}_L, \statevec{u}_R \in \mathcal{U}_{wb},
\end{equation}
which demonstrates that the EC fluctuations exactly preserve the steady-state solution.
What remains, is to verify that the blended dissipation term $\statemat{Q}_{ES}\jump{\statevec{u}}$ preserves this steady-state as well. 
For this, we consider the Roe matrix evaluated under lake-at-rest conditions.
We further assume that both the active sediment height $h_b$ and the sediment discharge vanish with the velocity. \changeTwo{The assumption holds for both the models of Grass and MPM considered in this work, but is also expected to hold for most other empirical models, where sediment transport is governed by shear stress. With this assumption the} Roe matrix then simplifies to be
	\begin{equation}
	\statemat{A}_{roe}(\statevec{u}_L, \statevec{u}_R) = \begin{pmatrix}
		0 & 1 & 0 \\ g\tilde{h} & 0 & g\tilde{h} \\ 0 & 0 & 0
	\end{pmatrix}, \quad
	\forall \statevec{u}_L, \statevec{u}_R \in \mathcal{U}_{wb}.
\end{equation}
Next, we introduce the matrix sign function $\text{sgn}(\statemat{\Lambda}) = \text{diag}(\text{sgn}(\lambda_1), ..., \text{sgn}(\lambda_m))$ and define $\text{sgn}(\statemat{A}) := \statemat{R}\text{sgn}(\statemat{\Lambda})\statemat{R}^{-1}$ to rewrite the dissipation term and obtain
\begin{equation}
	|\statemat{A}_{roe}|\jump{\statevec{u}} = 
	\text{sgn}(\statemat{A}_{roe})\statemat{A}_{roe}\jump{\statevec{u}} =
	\text{sgn}(\statemat{A}_{roe})\tilde{h}\jump{h + b} =  0, \quad \forall \statevec{u}_L, \statevec{u}_R \in \mathcal{U}_{wb}.
\end{equation}
Thus, the Roe dissipation term $\statemat{Q}_{roe}\jump{\statevec{u}}$ preserves the lake-at-rest steady state.
Well-balancedness for $\statevec{D}^{\pm}_{ES}$ then follows directly, as the entropy production  $\jump{\statevec{w}}^T\statemat{Q}_{roe}\jump{\statevec{u}}$ vanishes and consequently no LLF dissipation is introduced.

\section{Results}\label{sec:results}
We present numerical results to confirm the theoretical properties established in Section~\ref{sec:numerical_method} that the DGSEM with flux differencing achieves high-order accuracy, entropy stability, and well-balancedness when applied to the nonconservative SVE system from Section~\ref{sec:SVE_system}.

We validate and assess performance for EC fluctuations constructed from the two strategies outlined in Section~\ref{sec:numerical_method}.
For this we consider EC fluctuations $\statevec{D}_{EC,1}^{\pm}$ constructed from \eqref{eq:definition_ec_flux_strategy_1} using $N$-point Legendre-Gauss quadrature for the path integrals and the closed-form EC fluctuation $\statevec{D}_{EC,2}^{\pm}$ from \eqref{eq:ec_fluctuation_sve}.

Numerical results are obtained using the solver framework provided by Trixi.jl~\cite{schlottkelakemper2021purely} and TrixiShallowWater.jl~\cite{winters2025trixi} for the semi-discretization, using a 3rd order ESDIRK method \cite{kennedy2001additive} with adaptive time stepping from DifferentialEquations.jl~\cite{rackauckas2017differentialequations} for time integration.
Visualizations are generated with Makie.jl~\cite{Makie2021}.
For all test cases we set the gravitational acceleration $g=9.81$, sediment and fluid densities $\rho_f = 1.0$ and $\rho_s = 0.3$, sediment porosity $\varphi = 0.4$ and use the sediment transport formula from Grass~\cite{grass1981sediment} with coefficient $A_g = 0.01$.
% Reproducibility repo
A reproducibility repository containing the necessary code and information to reproduce the numerical results in this section is available on Zenodo \cite{ersing2026exnerRepro}.

\subsection{Convergence test}
To demonstrate spatial convergence, we compare the numerical solution against an exact reference solution constructed using the method of manufactured solutions. 
The reference solution is defined using trigonometric functions to ensure sufficient smoothness for testing high-order accuracy and is given by
\begin{equation}
	h(x,t)+b(x,t) = 4 + \cos(2\sqrt{2}\pi x) \cos(2 \pi t), \quad
	v(x,t) = 0.5, \quad
	b(x,t) = 1 + \sin(2\sqrt{2} \pi x).
\end{equation}
Table~\ref{tab:convergence_results_N3} presents convergence results at final time $t=1.0$, obtained on a Cartesian mesh in the domain $[0, \sqrt{2}]$ with periodic boundary conditions, using $\mathbb{P}^3$ polynomials and a fixed time step $\Delta t = {10}^{-3}$.
The results are obtained with ES fluctuations at element interfaces, constructed from EC fluctuations in the path integral form $\statevec{D}_{EC,1}^{\pm}$ and explicit form $\statevec{D}_{EC,2}^{\pm}$ and a LLF dissipation term.
\begin{table}[b]\label{tab:convergence_results_N3}
	\caption{Convergence results obtained for $\mathbb{P}^3$ polynomials at final time $t=1$ with fixed time step $\Delta t = {10}^{-3}$.}
	\centering
	\resizebox{\columnwidth}{!}{%
		\begin{tabular}{|c|cc|cc|cc|cc|cc|cc|}
			\toprule
			& \multicolumn{6}{|c|}{$\statevec{D}_{EC,1} + \statemat{Q}_{llf}$} & \multicolumn{6}{|c|}{$\statevec{D}_{EC,2} + \statemat{Q}_{llf}$} \\\hline
			$N_{elem}$ & $L^2(h)$ & EOC(h) & $L^2(hv)$ & EOC(hv) & $L^2(b)$ & EOC(b) & $L^2(h)$ & EOC(h) & $L^2(hv)$ & EOC(hv) & $L^2(b)$ & EOC(b) \\\midrule
			$2^2$ & $3.99\cdot10^{-2}$ & $-$ & $1.47\cdot10^{-1}$ & $-$ & $5.73\cdot10^{-3}$ & $-$ & $1.76\cdot10^{-2}$ & $-$ & $3.97\cdot10^{-2}$ & $-$ & $5.73\cdot10^{-3}$ & $-$\\
			$2^3$ & $2.02\cdot10^{-3}$ & $4.31$ & $1.22\cdot10^{-2}$ & $3.59$ & $2.78\cdot10^{-4}$ & $4.37$ & $1.77\cdot10^{-3}$ & $3.31$ & $1.17\cdot10^{-2}$ & $1.77$ & $2.80\cdot10^{-4}$ & $4.36$\\
			$2^4$ & $1.72\cdot10^{-4}$ & $3.56$ & $9.58\cdot10^{-4}$ & $3.67$ & $1.78\cdot10^{-5}$ & $3.96$ & $1.11\cdot10^{-4}$ & $4.00$ & $7.23\cdot10^{-4}$ & $4.01$ & $1.78\cdot10^{-5}$ & $3.98$\\
			$2^5$ & $1.08\cdot10^{-5}$ & $3.99$ & $5.86\cdot10^{-5}$ & $4.03$ & $1.12\cdot10^{-6}$ & $3.99$ & $6.95\cdot10^{-6}$ & $3.99$ & $4.49\cdot10^{-5}$ & $4.01$ & $1.12\cdot10^{-6}$ & $3.99$\\
			$2^6$ & $6.77\cdot10^{-7}$ & $4.00$ & $3.64\cdot10^{-6}$ & $4.01$ & $7.16\cdot10^{-8}$ & $3.97$ & $4.35\cdot10^{-7}$ & $4.00$ & $2.80\cdot10^{-6}$ & $4.00$ & $7.00\cdot10^{-8}$ & $4.00$\\
			\bottomrule
	\end{tabular}}
\end{table}
The results confirm that both EC fluctuation types achieve the expected fourth-order accuracy in all quantities, confirming the theoretical results.

\subsection{Channel flow problem}
Next, we investigate the channel flow problem from \cite{hudson2001numerical,hudson2003formulations}.
The test problem considers a channel on the domain $\Omega = [0, 1000]$ with initial conditions
\begin{equation}
	\begin{aligned}
		h(x,0) = 10 - b(x,0), \quad hv(x,0) = 10, \quad
		b(x,0) = \begin{cases} 
			\sin^2\left(\frac{\pi (x - 300)}{200}\right) & \text{if } \, 300 \leq x \leq 500,\\
			0 & \text{otherwise}.
		\end{cases}
	\end{aligned}
\end{equation}
An approximate closed-form solution using the method of characteristics was derived in \cite[Section~3.5.1]{hudson2001numerical} assuming rigid-lid conditions.
For the numerical solution, we discretize the domain $\Omega$ into $128$ equidistant elements, use $\mathbb{P}^4$ polynomials, and consider periodic boundary conditions to \changeTwo{ensure a closed system, where the total entropy is conserved.
Alternatively, inflow / outflow boundary conditions could be considered here.
However, to the best of the authors knowledge, the construction of provably stable open boundary conditions for the SVE system is still an open research topic.}

\subsubsection{Entropy conservation}
First, we verify the entropy conservation properties of the scheme and compare the accuracy and performance for different EC fluctuations.
To this end, we compute the channel flow problem up to a final time $t = 30,000$.

Figure~\ref{fig:time_evolution_channel_ec} shows the time evolution of the sediment height for numerical solutions obtained with EC fluctuations $\statevec{D}_{EC,1}$ and $\statevec{D}_{EC,2}$, along with the approximate reference solution from \cite{hudson2001numerical}.
\begin{figure}
	\centering
	\includegraphics[width=0.8\textwidth, trim={0 1cm 0 1cm},clip]{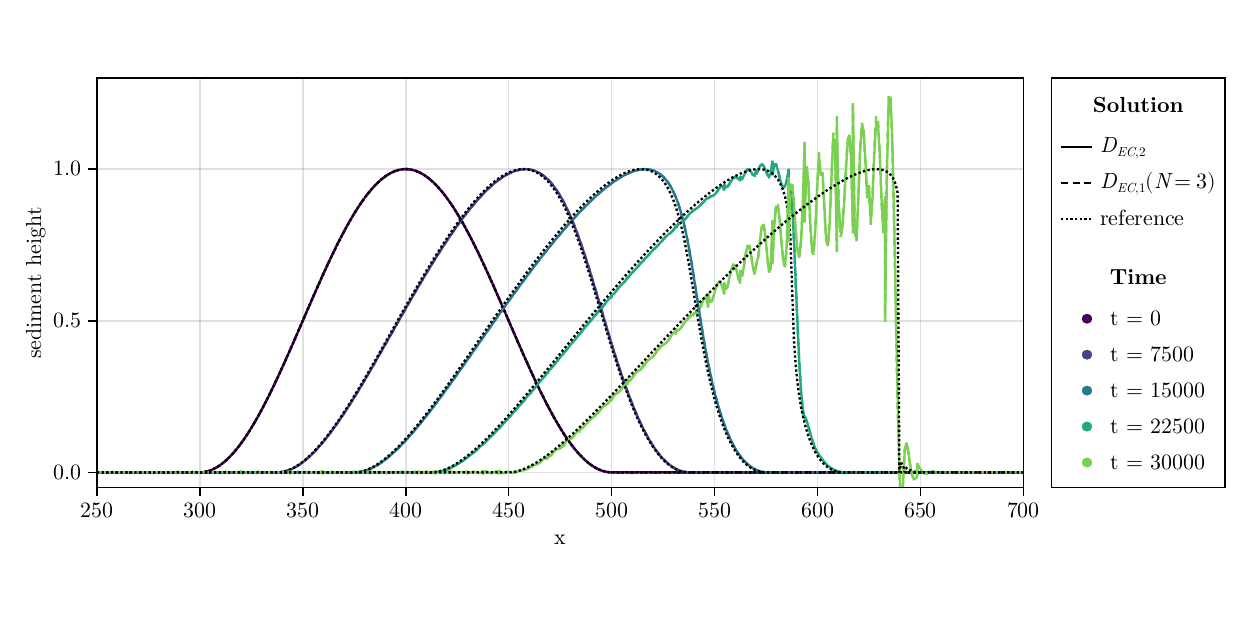}
	\caption{Time evolution of the sediment height $b$ in the channel flow test case, showing the approximate reference solution and numerical results obtained with $\mathbb{P}^4$ polynomials using EC fluctuations $\statevec{D}_{EC,1}$ and $\statevec{D}_{EC,2}$.}
	\label{fig:time_evolution_channel_ec}
\end{figure}
The figure illustrates the downstream propagation and gradual steepening of the sinusoidal sediment dune until a discontinuity forms around $t\approx23,800$.
For smooth solutions, the numerical results demonstrate a good agreement between the numerical and reference solution, with no visible difference between the two EC fluctuations.
After the onset of the discontinuity, the numerical solutions still \changeThree{accurately} predict the \changeThree{shock location provided by the approximate reference solution}, but exhibit spurious oscillations due to the absence of entropy dissipation.

We further compare numerical results obtained with the EC fluctuation $\statevec{D}_{EC,2}$ and EC fluctuations $\statevec{D}_{EC,1}(N)$ constructed following the strategy in Section~\ref{sec:ec_strategy_1} with a $N$-point Legendre-Gauss quadrature rule, using $N = \{1,2,3\}$ points.
Table~\ref{tab:comparison_channel_ec} reports the maximum entropy rate, the $L_2$ errors with respect to the reference solution, and the average computational time $\bar{t}_{rhs}$ required for a single evaluation of the time derivative in \eqref{eq:flux_diff_DGSEM} for each fluctuation.
\begin{table}
	\centering
	\caption{Comparison of the maximum entropy rate, convergence errors and computation time for the channel flow test case evaluated with different EC fluctuations.}
	\label{tab:comparison_channel_ec}
	\begin{tabular}{|l|c|ccc|c|}
		\toprule
		Fluctuation & $\max\limits_t \frac{1}{|\Omega|}\int S_t \, d\Omega$ & $||h - h_{ref}||_{L_2}^2$ & $||hv - hv_{ref}||_{L_2}^2$ & $||b - b_{ref}||_{L_2}^2$ & $\bar{t}_{rhs}$ \\
		\midrule
		$D_{EC,1}(N=1)$ & $1.370 \cdot 10^{-6}$  & $3.594 \cdot 10^{-2}$ & $1.571 \cdot 10^{-3}$ & $3.578 \cdot 10^{-2}$ & $306\,\mu s$ \\ 
		$D_{EC,1}(N=2)$ & $1.858 \cdot 10^{-11}$ & $3.777 \cdot 10^{-2}$ & $4.584 \cdot 10^{-3}$ & $3.739 \cdot 10^{-2}$ & $517\,\mu s$ \\
		$D_{EC,1}(N=3)$ & $3.496 \cdot 10^{-17}$ & $3.777 \cdot 10^{-2}$ & $4.584 \cdot 10^{-3}$ & $3.739 \cdot 10^{-2}$ & $738\,\mu s$ \\
		$D_{EC,2}$      & $1.682 \cdot 10^{-14}$ & $3.578 \cdot 10^{-2}$ & $7.430 \cdot 10^{-3}$ & $3.561 \cdot 10^{-2}$ & $47\,\mu s$ \\
		\bottomrule
	\end{tabular}
\end{table}

The $L_2$ errors demonstrate that for the considered test case different EC fluctuations produce nearly identical solutions with good agreement to the reference.
However, there is a large contrast in computational cost, with quadrature based fluctuations $\statevec{D}_{EC,1}(N)$ being between $~6-15$ slower than the closed-form fluctuation $\statevec{D}_{EC,2}$ for this configuration. 
The maximum entropy rate results further verify that the $\statevec{D}_{EC,2}$ fluctuation conserves entropy to machine precision in the semi-discrete case, while entropy conservation for $\statevec{D}_{EC,1}$ fluctuations depends on the accuracy of the quadrature rule used to compute the path integral. 
For $N \leq 2$ there is a noticeable violation of entropy conservation that vanishes to machine precision for $N\geq3$, highlighting that care must be taken to choose a quadrature rule of sufficient accuracy\changeTwo{, as discussed in Section~\ref{sec:ec_strategy_1}}.

These results confirm that both strategies to construct EC fluctuations are demonstrably EC and yield nearly identical solutions. 
However, the closed-form fluctuation $\statevec{D}_{EC,2}$ is significantly more computationally efficient and should therefore be the preferred whenever a derivation is possible.

\subsubsection{Entropy stability}
We again, compute the channel flow test case, now employing an ES formulation using the ES fluctuations $\statevec{D}_{EC,2} + \statemat{Q}_{llf}$ and $\statevec{D}_{EC,2} + \statemat{Q}_{ES-Roe}$ at element interfaces to assess the influence of different dissipation terms and confirm the ES properties.

Figure \ref{fig:time_evolution_channel_es} compares the time evolution of the sediment height obtained with both ES fluctuations to the approximate reference solution from \cite{hudson2001numerical}.
\begin{figure}
	\centering
	\includegraphics[width=0.8\textwidth, trim={0 1cm 0 1cm},clip]{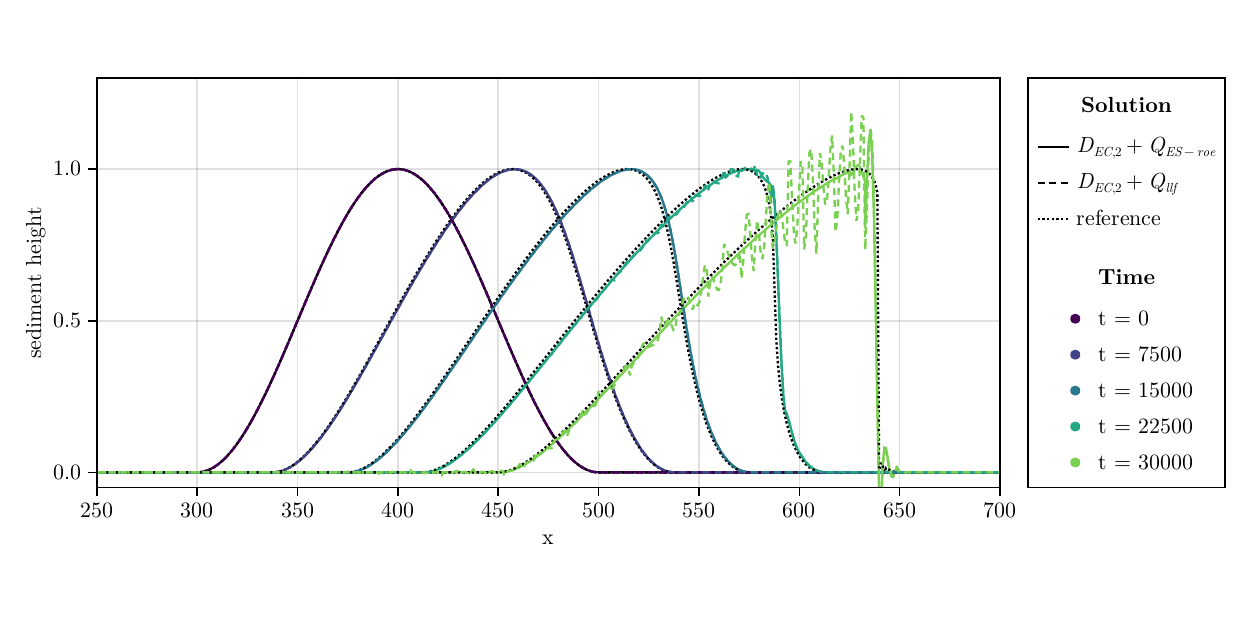}
	\caption{Time evolution of the sediment height $b$ in the channel flow test case, showing the approximate reference solution and numerical solutions results with $\mathbb{P}^4$ polynomials and ES fluctuations at interfaces.}
	\label{fig:time_evolution_channel_es}
\end{figure}
For the ES fluctuation with blended Roe-dissipation, we observe good agreement with the reference solution.
Only minor overshoots appear near the discontinuity, which is expected as we use a fifth order method without additional shock-capturing. 
While shock-capturing should be employed for practical applications, we intentionally avoid such techniques, as they dampen the underlying numerical instabilities we aim to investigate.

The numerical results obtained with LLF dissipation also agree well for smooth solutions.
However, once the discontinuity starts to form, the LLF dissipation fails to suppress the spurious oscillations that were already visible for the EC fluctuations in Figure~\ref{fig:time_evolution_channel_ec}.

To further analyze this behavior, we examine the evolution of total entropy over time for both variants of the ES fluctuations, as shown in Figure~\ref{fig:time_series_entropy}.
\begin{figure}
	\centering
	\includegraphics[width = 0.8\textwidth, trim={0 0 0 0},clip]{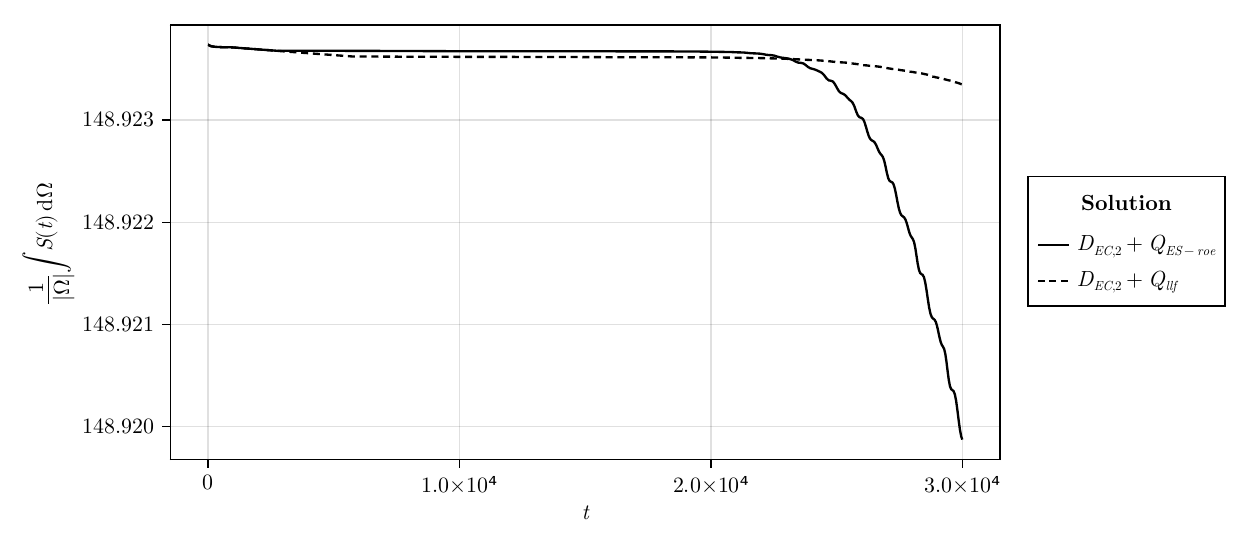}
	\caption{Total entropy over time for the channel flow test case with different ES fluctuations at interfaces.}
	\label{fig:time_series_entropy}
\end{figure}
The results demonstrate that total entropy decreases monotonically for either choice, confirming the expected entropy stability properties.
Moreover, the entropy evolution provides insight into why the LLF dissipation fails to suppress the oscillations in Figure~\ref{fig:time_evolution_channel_es}.
Once the shock forms, the blended Roe dissipation induces a significant entropy decrease, while the LLF solution exhibits only a minor decrease.
This observed behavior is rather unexpected, as the LLF dissipation is generally considered more dissipative than Roe-type dissipation.
However, the entropy evolution indicates that LLF dissipation provides insufficient dissipation near discontinuities, which causes the onset of oscillations for this test case.
The results demonstrate that the blended Roe dissipation provides more robust and accurate numerical approximations in the presence of discontinuities and should therefore be preferred over LLF dissipation in such regimes. 
\changeTwo{An analysis of the blending coefficients defined by \eqref{eq:definition_alpha} shows that for this test case, the interface blending procedure is activated only $50$ times over approximately $1\cdot10^6$ interface flux evaluations, with a maximum blending coefficient below $5\cdot10^{-4}$.
While these activations demonstrate that the additional blending is necessary to maintain entropy stability in practice, the actual correction is very minor in this example.
However, entropy violations generally depend strongly on the specified equations system, dissipation term, and test case.
Thus, these results should be regarded as a single observation rather than as representative behavior for the blending procedure.}

\subsection{Well-balancedness}
Finally, we consider a case with a discontinuous sediment layer that is initialized with the following lake-at-rest condition
\begin{equation}
	H = 0.5, \quad v = 0, \quad b = \begin{cases}
		0.4 & \text{for} \quad |x|< 0.5 \\ 0 & \text{otherwise,}
	\end{cases}
\end{equation}
on the domain $\Omega = [-2, 2]$. This configuration is used to verify the well-balanced property of our numerical method. 

Numerical solutions are computed on a spatial discretization of $16$ equidistant elements up to a final time $t=10$ with fixed time step $\Delta t = 2\cdot10^{-2}$.
In Figure~\ref{fig:well-balanced_results}, we present results for the ES fluctuation with blended Roe dissipation and compare it to the ES fluctuation constructed with LLF dissipation that is known to violate the well-balanced property.
\begin{figure}
	\centering
	\includegraphics[width=\textwidth, trim={0 1.5cm 0 0},clip]{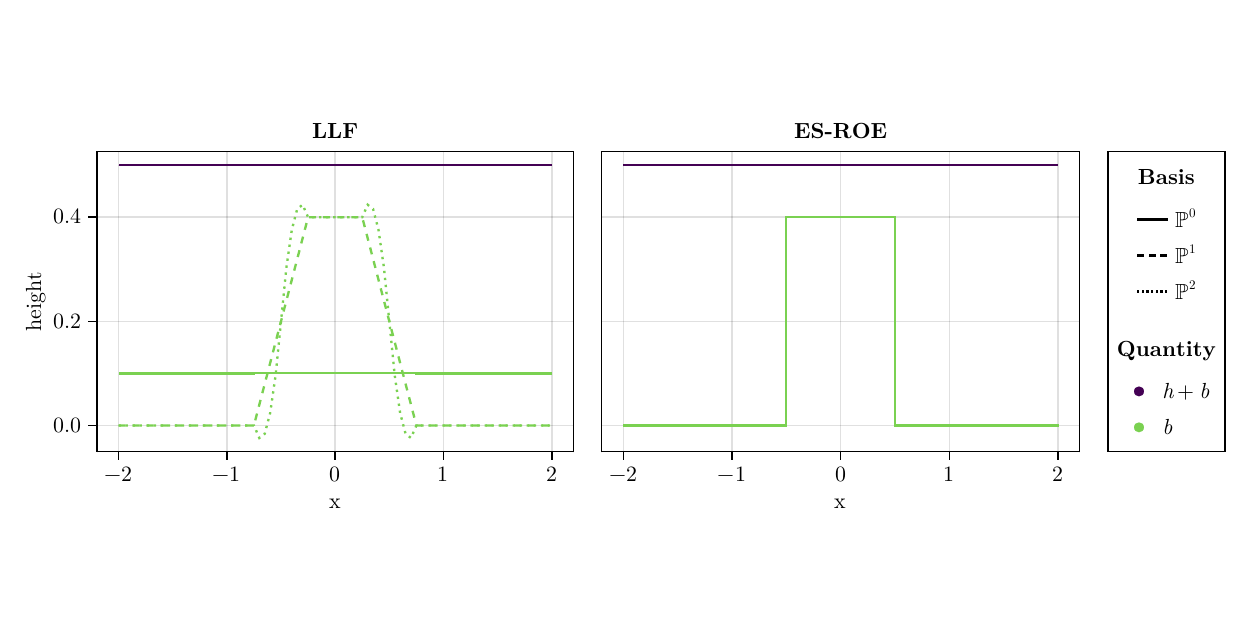}
	\caption{Water and sediment height obtained with ES fluctuations $\statevec{D}_{EC,2} + \statemat{Q}_{llf}$ (left) and $\statevec{D}_{EC,2} + \statemat{Q}_{ES-roe}$ (right) for the well-balanced test case at final time $t=10$. Results are shown for polynomial degrees $\mathbb{P}^0$, $\mathbb{P}^1$, and $\mathbb{P}^2$.}
	\label{fig:well-balanced_results}
\end{figure}

% Results description
The results demonstrate that the blended Roe dissipation exactly preserves the steady-state initial condition for all polynomial degrees, matching our theoretical expectations.
In contrast, the numerical solution with LLF dissipation fails to preserve the initial condition, due to spurious dissipation in the sediment equation.
This effect is most severe for the $\mathbb{P}^0$ approximation, where the initial sediment topography is completely dissipated, converging to a different steady-state solution with constant sediment topography.
For higher-order approximations, the main features of the initial sediment topography are preserved, although the sediment is still gradually dissipated until a continuous topography in the polynomial solution space is reached.

% Conclusion
Altogether, the results show that while the fluctuation with LLF dissipation may provide reasonable approximations for high-order polynomials, it fails to maintain the lake-at-rest equilibrium for zeroth order polynomials.
On the other hand, the blended Roe dissipation, exactly preserves the discontinuous lake-at-rest equilibrium for all polynomial degrees.

\section{Conclusion}\label{sec:conclusion}
In this work we extended the nonconservative discontinuous Galerkin spectral element method (DGSEM) from \cite{renac2019entropy} with a new very general class of entropy stable (ES) fluctuations.
We provided two strategies to construct entropy conservative (EC) fluctuations either explicitly without requiring any model-specific derivations or to construct very efficient EC fluctuations when a closed-form solution can be found.
Furthermore, we demonstrated the loss of entropy symmetrization for nonconservative systems and how this restricts the design of ES fluctuations.
As a remedy, we proposed a novel blending procedure that allows us to construct ES fluctuations from general matrix dissipation terms.
The resulting methodology was applied to develop a high-order, entropy stable and well-balanced DGSEM for the Saint-Venant-Exner system.
Finally, numerical results verified the analysis and demonstrated convergence, entropy stability and well-balancedness of the scheme.
		
\section*{CRediT authorship contribution statement}
% See https://credit.niso.org/
{\bf Patrick Ersing:} Conceptualization; Formal analysis; Investigation; Methodology; Visualization; Software; Writing - original draft

{\bf Andrew R. Winters:} Conceptualization; Funding acquisition; Supervision; Software; Writing - review \& editing

\section*{Declaration of generative AI use}
During the preparation of this work the authors used OpenAi Codex in order to assist with data visualizations and to improve structure and language. After using this tool, the authors reviewed and edited the content as needed and take full responsibility for the content of the published article.

\section*{Data Availability}
A reproducibility repository with necessary instructions and code to reproduce the presented results is available on Zenodo \cite{ersing2026exnerRepro}.
	
\section*{Declaration of competing interest}
The authors declare that they have no known competing financial interests or personal relationships that could have appeared to influence the work reported in this paper.
\section*{Acknowledgement}	
	Funding: This work was supported by Vetenskapsrådet, Sweden [grant agreement 2020-03642 VR]	

% Bibliography
\bibliographystyle{elsarticle-num}
\biboptions{sort&compress}	% sort & compress citations
\bibliography{bibliography.bib}

@article{wintermeyer2017entropy,
	title={An entropy stable nodal discontinuous {G}alerkin method for the two dimensional shallow water equations on unstructured curvilinear meshes with discontinuous bathymetry},
	author={Wintermeyer, Niklas and Winters, Andrew R and Gassner, Gregor J and Kopriva, David A},
	journal={Journal of Computational Physics},
	volume={340},
	pages={200--242},
	year={2017},
	publisher={Elsevier}
}

@article{gassner2021novel,
  title={A {n}ovel robust strategy for discontinuous {G}alerkin methods in computational fluid mechanics: {W}hy? {W}hen? {W}hat? {W}here?},
  author={Gassner, Gregor J and Winters, Andrew R},
  journal={Frontiers in Physics},
  volume={8},
  pages={500690},
  year={2021},
  publisher={Frontiers Media SA}
}

@book{kopriva2009implementing,
  title={Implementing spectral methods for partial differential equations: {A}lgorithms for scientists and engineers},
  author={Kopriva, David A},
  year={2009},
  publisher={Springer Science \& Business Media}
}

@article{gassner2013skew,
	title={A skew-symmetric discontinuous {G}alerkin spectral element discretization and its relation to {SBP-SAT} finite difference methods},
	author={Gassner, Gregor J},
	journal={SIAM Journal on Scientific Computing},
	volume={35},
	number={3},
	pages={A1233--A1253},
	year={2013},
	publisher={SIAM}
}

@article{tadmor1987numerical,
	title={The numerical viscosity of entropy stable schemes for systems of conservation laws. {I}},
	author={Tadmor, Eitan},
	journal={Mathematics of Computation},
	volume={49},
	number={179},
	pages={91--103},
	year={1987}
}

@article{chertock2018well,
	title={Well-balanced schemes for the shallow water equations with {C}oriolis forces},
	author={Chertock, Alina and Dudzinski, Michael and Kurganov, Alexander and Luk{\'a}{\v{c}}ov{\'a}-Medvid’ov{\'a}, M{\'a}ria},
	journal={Numerische Mathematik},
	volume={138},
	pages={939--973},
	year={2018},
	publisher={Springer}
}

@article{gassner2016split,
	title={Split form nodal discontinuous {G}alerkin schemes with summation-by-parts property for the compressible {E}uler equations},
	author={Gassner, Gregor J and Winters, Andrew R and Kopriva, David A},
	journal={Journal of Computational Physics},
	volume={327},
	pages={39--66},
	year={2016},
	publisher={Elsevier}
}

@article{fisher2013high,
	title={High-order entropy stable finite difference schemes for nonlinear conservation laws: Finite domains},
	author={Fisher, Travis C and Carpenter, Mark H},
	journal={Journal of Computational Physics},
	volume={252},
	pages={518--557},
	year={2013},
	publisher={Elsevier}
}

@article{carpenter2014entropy,
	title={Entropy stable spectral collocation schemes for the {N}avier--{S}tokes equations: {D}iscontinuous interfaces},
	author={Carpenter, Mark H and Fisher, Travis C and Nielsen, Eric J and Frankel, Steven H},
	journal={SIAM Journal on Scientific Computing},
	volume={36},
	number={5},
	pages={B835--B867},
	year={2014},
	publisher={SIAM}
}

@article{tadmor2003entropy,
	title={{E}ntropy stability theory for difference approximations of nonlinear conservation laws and related time-dependent problems},
	author={Tadmor, Eitan},
	journal={Acta Numerica},
	volume={12},
	pages={451--512},
	year={2003},
	publisher={Cambridge University Press}
}

@article{bohm2020entropy,
	title={{An entropy stable nodal discontinuous {G}alerkin method for the resistive {MHD} equations. {P}art {I}: {T}heory and numerical verification}},
	author={Bohm, Marvin and Winters, Andrew R and Gassner, Gregor J and Derigs, Dominik and Hindenlang, Florian and Saur, Joachim},
	journal={Journal of Computational Physics},
	volume={422},
	pages={108076},
	year={2020},
	publisher={Elsevier}
}

@article{schlottkelakemper2021purely,
	title={A purely hyperbolic discontinuous {G}alerkin approach for
	self-gravitating gas dynamics},
	author={Schlottke-Lakemper, Michael and Winters, Andrew R and
	Ranocha, Hendrik and Gassner, Gregor J},
	journal={Journal of Computational Physics},
	pages={110467},
	year={2021},
	month={06},
	volume={442},
	publisher={Elsevier},
	doi={10.1016/j.jcp.2021.110467},
	eprint={2008.10593},
	eprinttype={arXiv},
	eprintclass={math.NA}
}

@article{chan2022entropy,
	title={Entropy stable modal discontinuous {G}alerkin schemes and wall boundary conditions for the compressible {N}avier-{S}tokes equations},
	author={Chan, Jesse and Lin, Yimin and Warburton, Tim},
	journal={Journal of Computational Physics},
	volume={448},
	pages={110723},
	year={2022},
	publisher={Elsevier}
}

@article{franquet2012runge,
	title={Runge--{K}utta discontinuous {G}alerkin method for the approximation of {B}aer and {N}unziato type multiphase models},
	author={Franquet, Erwin and Perrier, Vincent},
	journal={Journal of Computational Physics},
	volume={231},
	number={11},
	pages={4096--4141},
	year={2012},
	publisher={Elsevier}
}

@article{dal1995definition,
	title={{D}efinition and weak stability of nonconservative products},
	author={Dal Maso, Gianni and Lefloch, Philippe G and Murat, Fran{\c{c}}ois},
	journal={Journal de math{\'e}matiques pures et appliqu{\'e}es},
	volume={74},
	number={6},
	pages={483--548},
	year={1995}
}

@article{waruszewski2022entropy,
	title={Entropy stable discontinuous {G}alerkin methods for balance laws in non-conservative form: {A}pplications to the {E}uler equations with gravity},
	author={Waruszewski, Maciej and Kozdon, Jeremy E and Wilcox, Lucas C and Gibson, Thomas H and Giraldo, Francis X},
	journal={Journal of Computational Physics},
	volume={468},
	pages={111507},
	year={2022},
	publisher={Elsevier}
}

@article{pares2006numerical,
	title={{N}umerical methods for nonconservative hyperbolic systems: {A} theoretical framework.},
	author={Par{\'e}s, Carlos},
	journal={SIAM Journal on Numerical Analysis},
	volume={44},
	number={1},
	pages={300--321},
	year={2006},
	publisher={SIAM}
}

@article{renac2019entropy,
	title={Entropy stable {DGSEM} for nonlinear hyperbolic systems in nonconservative form with application to two-phase flows},
	author={Renac, Florent},
	journal={Journal of Computational Physics},
	volume={382},
	pages={1--26},
	year={2019},
	publisher={Elsevier}
}

@article{castro2013entropy,
	title={{E}ntropy conservative and entropy stable schemes for nonconservative hyperbolic systems},
	author={Castro, Manuel J and Fjordholm, Ulrik S and Mishra, Siddhartha and Par{\'e}s, Carlos},
	journal={SIAM Journal on Numerical Analysis},
	volume={51},
	number={3},
	pages={1371--1391},
	year={2013},
	publisher={SIAM}
}

@article{coquel2021entropy,
	title={An entropy stable high-order discontinuous {G}alerkin spectral element method for the {B}aer-{N}unziato two-phase flow model},
	author={Coquel, Fr{\'e}d{\'e}ric and Marmignon, Claude and Rai, Pratik and Renac, Florent},
	journal={Journal of Computational Physics},
	volume={431},
	pages={110135},
	year={2021},
	publisher={Elsevier}
}

@article{merriam1989entropy,
	title={An {E}ntropy-{B}ased {A}pproach to {N}onlinear {S}tability},
	author={Merriam, Marshal L},
	journal={NASA Technical Memorandum},
	volume={101},
	pages={086},
	year={1989},
	publisher={Citeseer}
}

@article{Makie2021,
	doi = {10.21105/joss.03349},
	xurl= {https://doi.org/10.21105/joss.03349},
	year = {2021},
	publisher = {The Open Journal},
	volume = {6},
	number = {65},
	pages = {3349},
	author = {Simon Danisch and Julius Krumbiegel},
	title = {{Makie.jl}: {F}lexible high-performance data visualization for {Julia}},
	journal = {Journal of Open Source Software}
}

@article{rhebergen2008discontinuos,
	title = {Discontinuous {G}alerkin finite element methods for hyperbolic nonconservative partial differential equations},
	journal = {Journal of Computational Physics},
	volume = {227},
	number = {3},
	pages = {1887-1922},
	year = {2008},
	issn = {0021-9991},
	doi = {https://doi.org/10.1016/j.jcp.2007.10.007},
	author = {S. Rhebergen and O. Bokhove and J.J.W. {van der Vegt}},
}

@book{ranocha2018generalised,
	title={Generalised summation-by-parts operators and entropy stability of numerical methods for hyperbolic balance laws},
	author={Ranocha, Hendrik},
	year={2018},
	publisher={Cuvillier Verlag}
}

@article{cordesse2019entropy,
	title={Entropy supplementary conservation law for non-linear systems of {PDE}s with non-conservative terms: application to the modelling and analysis of complex fluid flows using computer algebra},
	author={Cordesse, Pierre and Massot, Marc},
	journal={arXiv preprint arXiv:1911.02313},
	year={2019}
}

@article{schmidtmann2017hybrid,
	title={Hybrid entropy stable {HLL}-type {R}iemann solvers for hyperbolic conservation laws},
	author={Schmidtmann, Birte and Winters, Andrew R},
	journal={Journal of Computational Physics},
	volume={330},
	pages={566--570},
	year={2017},
	publisher={Elsevier}
}

@incollection{barth1999numerical,
	title={Numerical methods for gasdynamic systems on unstructured meshes},
	author={Barth, Timothy J},
	booktitle={An Introduction to Recent Developments in Theory and Numerics for Conservation Laws: Proceedings of the International School on Theory and Numerics for Conservation Laws, Freiburg/Littenweiler, October 20--24, 1997},
	pages={195--285},
	year={1999},
	publisher={Springer}
}

@inproceedings{godunov1961interesting,
	title={An interesting class of quasilinear systems},
	author={Godunov, Sergei Konstantinovich},
	booktitle={Dokl. Akad. Nauk SSSR},
	volume={139},
	pages={521--523},
	year={1961}
}

@article{mock1980systems,
	title={Systems of conservation laws of mixed type},
	author={Mock, M.},
	journal={Journal of Differential equations},
	year={1980},
	publisher={Elsevier}
}

@article{fernandez2017formal,
	title={Formal deduction of the {S}aint-{V}enant--{E}xner model including arbitrarily sloping sediment beds and associated energy},
	author={Fern{\'a}ndez-Nieto, Enrique D and Luna, Tom{\'a}s Morales de and Narbona-Reina, Gladys and Zabsonr{\'e}, Jean de Dieu},
	journal={ESAIM: Mathematical Modelling and Numerical Analysis},
	volume={51},
	number={1},
	pages={115--145},
	year={2017}
}

@article{diaz2009two,
	title={Two-dimensional sediment transport models in shallow water equations. {A} second order finite volume approach on unstructured meshes},
	author={Castro, Manuel J and Fern{\'a}ndez-Nieto, Enrique D and Ferreiro, AM and Par{\'e}s, C and others},
	journal={Computer Methods in Applied Mechanics and Engineering},
	volume={198},
	number={33-36},
	pages={2520--2538},
	year={2009},
	publisher={Elsevier}
}

@article{carraro2018efficient,
	title={Efficient analytical implementation of the {DOT} {R}iemann solver for the de {S}aint {V}enant--{E}xner morphodynamic model},
	author={Carraro, Francesco and Valiani, Alessandro and Caleffi, Valerio},
	journal={Advances in water resources},
	volume={113},
	pages={189--201},
	year={2018},
	publisher={Elsevier}
}

@phdthesis{hudson2001numerical,
	title={Numerical techniques for morphodynamic modelling},
	author={Hudson, Justin},
	year={2001},
	school={Citeseer}
}

@article{cordier2011bedload,
	title={Bedload transport in shallow water models: {W}hy splitting (may) fail, how hyperbolicity (can) help},
	author={Cordier, St{\'e}phane and Le, Minh H and De Luna, T Morales},
	journal={Advances in Water Resources},
	volume={34},
	number={8},
	pages={980--989},
	year={2011},
	publisher={Elsevier}
}

@book{grass1981sediment,
	title={Sediment transport by waves and currents},
	author={Grass, Arnold Jules},
	year={1981},
	publisher={University College, London, Department of Civil Engineering}
}

@techreport{meyer1948formulas,
	title        = {Formulas for {B}ed-{L}oad {T}ransport},
	author       = {Meyer-Peter, E. and Müller, R.},
	year         = {1948},
	institution  = {{International Association for Hydraulic Research}},
	address      = {Stockholm},
	type         = {Proceedings of the 2nd Meeting},
	xurl         = {https://resolver.tudelft.nl/uuid:4fda9b61-be28-4703-ab06-43cdc2a21bd7},
}

@book{kennedy2001additive,
      title={Additive {R}unge-{K}utta schemes for convection-diffusion-reaction equations},
	author={Kennedy, Christopher Alan},
	year={2001},
	publisher={National Aeronautics and Space Administration, Langley Research Center}}

@article{rackauckas2017differentialequations,
	title={Differential{E}quations.jl--a performant and feature-rich ecosystem for solving differential equations in {J}ulia},
	author={Rackauckas, Christopher and Nie, Qing},
	journal={Journal of Open Research Software},
	volume={5},
	number={1},
	year={2017},
	publisher={Ubiquity Press}
}

@misc{winters2025trixi,
	title={{TrixiShallowWater.jl}: {S}hallow water simulations with {T}rixi.jl},
	author={Winters, Andrew R and Ersing, Patrick and Ranocha, Hendrik and Schlottke-Lakemper, Michael},
	year={2025},
	howpublished={\url{https://github.com/trixi-framework/TrixiShallowWater.jl}},
	doi={10.5281/zenodo.15206520}
}

@article{Ferrer2023,
	title = {{HORSES3D}: {A} high-order discontinuous {G}alerkin solver for flow simulations and multi-physics applications},
	journal = {Computer Physics Communications},
	volume = {287},
	pages = {108700},
	year = {2023},
	issn = {0010-4655},
	doi = {https://doi.org/10.1016/j.cpc.2023.108700},
	xurl= {https://www.sciencedirect.com/science/article/pii/S0010465523000450},
	author = {E. Ferrer and G. Rubio and G. Ntoukas and W. Laskowski and O.A. Mariño and S. Colombo and A. Mateo-Gabín and H. Marbona and F. {Manrique de Lara} and D. Huergo and J. Manzanero and A.M. Rueda-Ramírez and D.A. Kopriva and E. Valero}
}

@software{ersing2026exnerRepro,
	title={Reproducibility repository for "{A} new class of entropy stable fluctuations for the discontinuous {G}alerkin method with application to the
	{S}aint-{V}enant-{E}xner model"},
	author={Ersing, Patrick and Winters, Andrew R},
	year={2026},
	publisher    = {Zenodo},
	doi          = {10.5281/zenodo.18173726},
}

@article{abgrall2010comment,
	title={A comment on the computation of non-conservative products},
	author={Abgrall, R{\'e}mi and Karni, Smadar},
	journal={Journal of Computational Physics},
	volume={229},
	number={8},
	pages={2759--2763},
	year={2010},
	doi = {10.1016/j.jcp.2009.12.015},
	publisher={Elsevier}
}

@article{hudson2003formulations,
	title={Formulations for numerically approximating hyperbolic systems governing sediment transport},
	author={Hudson, Justin and Sweby, Peter K},
	journal={Journal of Scientific Computing},
	volume={19},
	number={1},
	pages={225--252},
	year={2003},
	doi = {10.1023/A:1025304008907},
	publisher={Springer}
}

\appendix
\section{Equivalence of left and right symmetrization}\label{app:lemma4}
\begin{lemma}\label{lemma:left_right_symmetrization}
	Let $\statemat{H} \in \mathbb{R}^{n \times n}$ be symmetric positive definite (SPD) and $\statemat{A} \in \mathbb{R}^{n \times n}$. Then $\statemat{H}^{-1}$ symmetrizes $\statemat{A}$ from the left, if and only if $\statemat{H}$ symmetrizes $\statemat{A}$ from the right. 
\end{lemma}
\begin{proof}
	Since $\statemat{H}$ is SPD it follows that $\statemat{H}^{-1}$ exists and is also SPD. Then first assume that $\statemat{H}^{-1}$ symmetrizes $\statemat{A}$ from the left
	\begin{equation}
		\statemat{H}^{-1}\statemat{A} = (\statemat{H}^{-1}\statemat{A})^T = \statemat{A}^T\statemat{H}^{-1},
	\end{equation}
	then multiplication with $\statemat{H}$ from both left and right shows that $\statemat{H}$ is a right symmetrizer for $\statemat{A}$
	\begin{equation}
		\statemat{A}\,\statemat{H} = (\statemat{A}\,\statemat{H})^T = \statemat{H}\,\statemat{A}^T.
	\end{equation}
	To show the converse, assume $\statemat{H}$ symmetrizes $\statemat{A}$ from the right. Then multiplication with $\statemat{H}^{-1}$ from both left and right shows that $\statemat{H}$ is a left symmetrizer for $\statemat{A}$ which shows that both statements are equivalent.
	\qed
\end{proof}

\end{document}